\documentclass[a4paper,11pt,reqno]{amsart}

\usepackage{amsfonts}
\usepackage{amssymb}
\usepackage[francais,english]{babel}
\usepackage{amscd}
\usepackage{amsthm}
\usepackage{a4wide}
\usepackage{mathrsfs}
\usepackage[applemac]{inputenc}
\usepackage{amsmath}
\usepackage{amssymb}
\usepackage{amsthm}
\usepackage{textcomp}
\usepackage{graphicx}
\usepackage{enumerate}
\usepackage{mathrsfs}
\usepackage{frcursive}
\usepackage{tikz}
\usepackage[cyr]{aeguill}
\usepackage{xspace}
\usepackage{hyperref}
\usepackage{appendix}
\usepackage{esint}
\usepackage{cancel}
\usepackage{calc}
\usepackage{bm}

\newtheorem{defn}{Definition}[section]
\newtheorem{lemma}[defn]{Lemma}
\newtheorem{prop}[defn]{Proposition}
\newtheorem{theo}[defn]{Theorem}
\newtheorem{coro}[defn]{Corollary}

\newtheorem{claim}{Claim}
\newtheorem{rk}[defn]{Remark}

\def\eucl{\mathop{\rm eucl}\nolimits}

\def\div{\mathop{\rm div}\nolimits}

\def\eucl{\mathop{\rm eucl}\nolimits}

\def\div{\mathop{\rm div}\nolimits}

\def\supp{\mathop{\rm supp}\nolimits}

\def\R{\mathop{\rm \mathbb{R}}\nolimits}

\def\Lip{\mathop{\rm Lip}\nolimits}
\def\sgn{\mathop{\rm sgn}\nolimits}

\newsavebox\CBox
\newcommand\hcancel[2][0.5pt]{%
  \ifmmode\sbox\CBox{$#2$}\else\sbox\CBox{#2}\fi%
  \makebox[0pt][l]{\usebox\CBox}%
  \rule[0.5\ht\CBox-#1/2]{\wd\CBox}{#1}}

\newcommand{\nor}{\hcancel{\|}}

\begin{document}
\title{Existence of expanders of the harmonic map flow}
\date{\today}

\author{Alix Deruelle and Tobias Lamm}
\address[Alix Deruelle]{Institut de Math\'ematiques de Jussieu, Paris Rive Gauche (IMJ-PRG) UPMC - Campus Jussieu, 4, place Jussieu Bo\^ite Courrier 247 - 75252 Paris Cedex 05}
\email{alix.deruelle@imj-prg.fr}
\address[Tobias Lamm]{Institute for Analysis, Karlsruhe Institute of Technology (KIT), Englerstr. 2, 76131
Karlsruhe, Germany}
\email{tobias.lamm@kit.edu}

\begin{abstract}
 We investigate the existence of weak expanding solutions of the harmonic map flow for maps with values into a smooth closed Riemannian manifold. We prove the existence of such solutions in the case the target manifold is isometrically embedded as a hypersurface of some Euclidean space and the initial condition is a Lipschitz map that is homotopic to a constant. Regularity is proved outside a compact set.
\end{abstract}

\maketitle

\tableofcontents

\section{Introduction}
In this paper, we consider the Cauchy problem for the harmonic map flow of maps $(u(t))_{t\geq 0}$ from $\R^n$, $n\geq 3$ to a closed smooth Riemannian manifold $(N^{m-1},g)$ isometrically embedded as a hypersurface in some Euclidean space $\R^m$, $m\geq 2$. More precisely, we study the parabolic system
\begin{equation}
\left\{\begin{aligned}
&\partial_tu=\Delta u+A(u)(\nabla u,\nabla u),\quad\mbox{on $\mathbb{R}^n\times\mathbb{R}_+$},&\label{eq-HMP} \\
&u|_{t=0}=u_0,&
\end{aligned}
\right.
\end{equation}
for a given map $u_0:\mathbb{R}^n\rightarrow N$, where $A(u)(\cdot,\cdot):T_uN\times T_uN\rightarrow (T_uN)^{\perp}$ denotes the second fundamental form of the embedding $N^{m-1} \hookrightarrow \R^{m}$ evaluated at $u$. Note that the equation (\ref{eq-HMP}) is equivalent to $\partial_tu-\Delta u\perp T_uN$ for a family of maps $(u(t))_{t\geq0}$ which map into $N$.
Recall that this evolution equation is invariant under the scaling
\begin{eqnarray}
(u_{0})_{\lambda}(x)&:=&u_0(\lambda x),\quad x\in\R^n,\label{scal-1}\\
u_{\lambda}(x,t)&:=&u(\lambda x,\lambda^2t), \quad \lambda>0,\quad \text{$(x,t)\in \R^n\times\R_+$}.\label{resc-cond}
\end{eqnarray}
 
 If $u_0$ is invariant under the above scaling, i.e. if $u_0$ is $0$-homogeneous, solutions of the harmonic map flow which are invariant under scaling are potentially well-suited for smoothing out $u_0$ instantaneously. Such solutions are called expanding solutions or expanders.
 In this setting, it turns out that (\ref{eq-HMP}) is equivalent to a static equation, i.e. that it does not depend on time anymore. Indeed, if $u$ is an expanding solution in the previous sense then the map $U(x):=u(x,1)$ for $x\in \mathbb{R}^n$, satisfies the elliptic system
\begin{equation}
\left\{\begin{aligned}
&\Delta_f U+A(U)(\nabla U,\nabla U)=0,\quad\mbox{on $\mathbb{R}^n$},&\label{eq-HMP-Stat}\\
&\lim_{r\rightarrow+\infty}U(r,\omega)=u_0(\omega),\quad (r,\omega)\in \R_+\times\mathbb{S}^{n-1},
\end{aligned}
\right.
\end{equation}
 
where, $f$ and $\Delta_f$ are defined by
\begin{eqnarray*}
&&f(x):=\frac{|x|^2}{4}+\frac{n}{2},\quad x\in\mathbb{R}^n,\\
&&\Delta_fU:=\Delta U+\nabla f\cdot\nabla U=\Delta U+\frac{r}{2}\partial_rU.
\end{eqnarray*}
The function $f$ is called the potential function and it is defined up to an additive constant. The choice of this constant is dictated by the requirement $$\Delta_ff=f.$$ The operator $\Delta_f$ is called a weighted laplacian and it is unitarily conjugate to a harmonic oscillator $\Delta-|x|^2/16$.

Conversely, if $U$ is a solution to (\ref{eq-HMP-Stat}) then the map $u(x,t):=U(x/\sqrt{t})$, for $(x,t)\in\R^n\times\R_+$, is a solution to (\ref{eq-HMP}). Because of this equivalence, $u_0$ can be interpreted either as an initial condition or as a boundary data at infinity. 

The interest in expanding solutions is basically due to two reasons. On the one hand, these scale invariant solutions are important with respect to the continuation of a weak harmonic map flow between two closed Riemannian manifolds. Indeed, by the work of Chen and Struwe \cite{Chen-Struwe}, there always exists a weak solution of the harmonic map flow starting from a smooth map between two closed Riemannian manifolds. It turns out that such a flow is not always smooth and the appearance of singularities is caused by either non-constant $0$-homogeneous harmonic maps defined from $\mathbb{R}^n$ to $N$, the so called tangent maps, or shrinking solutions (also called quasi-harmonic spheres) that are ancient solutions invariant under scaling. Expanding solutions can create an ambiguity in the continuation of the flow after it reaches a singularity by a gluing process.  On the other hand, one might be interested in using the smoothing effect of the harmonic map flow. More precisely, it is tempting to attach a canonical map to any map between (stratified) manifolds with prescribed singularities. It turns out that $0$-homogeneous maps are the building blocks of such singularities and expanding solutions are likely to be the best candidates to do this job.

In this paper, we investigate the question of existence of expanding solutions coming out of $u_0$ in the case where there is no topological obstruction, i.e. under the assumption that $u_0$ is homotopic to a constant when restricted to $\mathbb{S}^{n-1}$. Our main result is the following

\begin{theo}\label{main-theo}
Let $n\geq 3$ and $m\geq 2$ be two integers and let $u_0:\mathbb{R}^n\rightarrow (N^{m-1},g)\subset\R^m$ be a Lipschitz $0$-homogeneous map such that its restriction to $\mathbb{S}^{n-1}$ is homotopic to a constant. 

Then there exists a weak expanding solution $u(\cdot,1)=:U(\cdot)$ of the harmonic map flow coming out of $u_0$ weakly which is regular off a closed singular set with at most finite $(n-2)$-dimensional Hausdorff measure. Moreover, there exists a radius $R=R(\|\nabla u_0\|_{L^2_{loc}},n,m)>0$ and a constant $C=C(\|\nabla u_0\|_{L^2_{loc}},n,m)>0$ such that $U$ is smooth outside $B(0,R)$ and,
\begin{eqnarray*}
&&|\nabla U|(x)\leq \frac{C}{|x|},\quad |x|\geq R,\\
&&\|\nabla u(t)\|_{L^2(B(x_0,1))}\leq C\left(n,m,\|\nabla u_0\|_{L^2_{loc}(\R^n)},t\right)\|\nabla u_0\|_{L^2(B(x_0,1))},\quad \forall x_0\in\R^n,\\
&&\|\partial_tu\|_{L^2((0,t),L^2_{loc}(\R^n))}\leq C(n,m,t)\|\nabla u_0\|_{L^2_{loc}(\R^n)},\\
\end{eqnarray*}
where $\lim_{t\rightarrow 0}C\left(n,m,\|\nabla u_0\|_{L^2_{loc}(\R^n)},t\right)=\lim_{t\rightarrow0}C(n,m,t)=1$.

In particular, $u(\cdot, t)$ tends to $u_0$ as $t$ goes to $0$ in the $H^1_{loc}(\R^n)$ sense and if $u_0$ is not harmonic then $u(\cdot,t)$ is not constant in time. Finally, one has the following convergence rate:
\begin{eqnarray}
|U(x)-u_0(x/|x|)|\leq C|x|^{-1},\quad |x|\geq R.\label{conv-rate}
\end{eqnarray}

\end{theo}

We remark that the regularity result of the theorem is reminiscent of and based on the fundamental work of Chen and Struwe \cite{Chen-Struwe}. We localise their approach to ensure the smoothness of the solution outside a closed ball since the local energy is decaying to $0$ at infinity. This lets us to establish a sharp convergence rate for Lipschitz maps. Theorem \ref{main-theo} and its proof provide the existence of a non constant in time (or equivalently non radial) expanding solution in case the initial map is not harmonic. Since the initial condition $u_0$ is allowed to have large local-in-space energy, it is likely that uniqueness would fail. In particular, the authors do not know if the solution produced by Theorem \ref{main-theo} coming out of a $0$-homogeneous harmonic map will stay harmonic.  

Let us make some comments about the proof of Theorem \ref{main-theo} before we describe its main steps. A direct perturbative approach is well-suited in case the target Riemannian manifold $(N,g)$ has non-positive sectional curvature as shown by the second author \cite{Der-rel-Ent-HMF} or if the $L^2_{loc}(\R^n)$ energy of $u_0$ is assumed to be arbitrarily small: see Section \ref{sec-def-theory-HMF} for a proof. One more instance where such a direct approach works well is by imposing further symmetry on the initial condition $u_0$ and the target manifold $N$ as initiated by Germain and Rupflin \cite{Ger-Rup}. To conclude, the nonlinearity of the target manifold and the potential formation of finite time singularities are the two main obstacles to a direct perturbative approach in general.  

To circumvent this issue, we follow Chen-Struwe's penalisation procedure \cite{Chen-Struwe} and we construct our expanding solution to the harmonic map flow as a limit of expanding solutions starting from the same initial condition $u_0$ of a so called homogeneous Chen-Struwe flow with parameter $K$, see Section \ref{Section-Fixed-point-formulation} for more definitions. Before stating the main result about this flow, we recall some definitions. 

Let $n\geq 3$ and let $u_0:\R^n\rightarrow (N,g)\subset\R^m$ be a $0$-homogeneous map. Let us notice that $(N,g)$ is not assumed to be a hypersurface of $\R^m$ at this stage.  

A map $u:\R^n\times(0,T)\rightarrow\R^m$ is a weak solution of the Homogeneous Chen-Struwe flow with initial condition $u_0$ if it satisfies:
\begin{equation}
\left\{\begin{aligned}
&\partial_tu=\Delta u-\frac{K}{t}\chi'\left( d^2_{N}(u)\right)\nabla\left( \frac{d^2_{N}}{2}\right)(u),\quad\mbox{on $\mathbb{R}^n\times\mathbb{R}_+$},&\\
&u|_{t=0}=u_0,\quad\text{in the weak sense},&\label{hom-chen-struwe-def-intro}
\end{aligned}
\right.
\end{equation}
where $\chi$ is a smooth, non-decreasing function such that $\chi(s)=s$ for $0\leq s\leq \delta_N^2$ and $\chi(s)=2\cdot \delta_N^2$ if $s\geq (2\cdot \delta_N)^2$ for some positive constant $\delta_N$ depending on the geometry of the embedding of $N$ in $\R^m$. As in \cite{Chen-Struwe}, such a function $\chi$ is designed to ignore the set of points of $\R^m\setminus N$ where the distance function $d_N$ is not smooth. We notice that (\ref{hom-chen-struwe-def-intro}) differs from the flow introduced by Chen and Struwe by a factor of time $t^{-1}$ in front of the parameter $K$: the main reason is to keep (\ref{hom-chen-struwe-def-intro}) invariant under parabolic rescalings as defined in (\ref{resc-cond}).
 
Moreover, the pointwise energy associated to a solution $u:\R^n\times\R_+\rightarrow \R^m$ of the Homogeneous Chen-Struwe flow (\ref{hom-chen-struwe-def-intro}) is defined by 
\begin{eqnarray}
e_K(u)(x,t):=\frac{1}{2}\left(|\nabla u|^2(x,t)+\frac{K}{t}\chi\left( d^2_{N}(u(x,t))\right)\right),\quad (x,t)\in\R^n\times\R_+.\label{ptwise-energy-K-intro}
\end{eqnarray}

For any expanding solution $u_K$ of the Homogeneous Chen-Struwe flow with parameter $K$ defined on $\R^n\times \R_+$, we denote its evaluation at time $t=1$ by $U_K(x):=u_K(x,1)$, for $x\in\R^n.$ \\

Formally speaking, the Homogeneous Chen-Struwe flow with parameter $K$ is the gradient flow of the (total) energy of a solution $u$ given by the integral of the pointwise energy $e_K$ introduced in (\ref{ptwise-energy-K-intro}). As $K$ goes to $+\infty$, the penalty term $$ K\int_{B(x_0,1)}\chi\left( d^2_{N}(u(x,t))\right)dx,\quad (x_0,t)\in\R^n\in\R_+,$$ forces the unconstrained solutions $u_K$ to take their values in $N$, at least heuristically.

The following theorem sums up the main properties of the  expanding solutions to this Homogeneous Chen-Struwe flow we are able to construct and make precise these heuristics:

\begin{theo}\label{Chen-Str-app}
Let $u_0:\mathbb{R}^n\rightarrow (N^{m-1},g)\subset \R^m$ be a $0$-homogeneous map in $C^3_{loc}(\R^n\setminus\{0\})$ for $n\geq 3$ such that its restriction to $\mathbb{S}^{n-1}$ is homotopically trivial. Then for any $K>0$, there exists a regular Chen-Struwe expanding solution $u_K$ coming out of $u_0$
\begin{equation}
\left\{\begin{aligned}
&-\Delta_fU_K+K\chi'\left( d^2_{N}(U_K)\right)\nabla\left( \frac{d^2_{N}}{2}\right)(U_K)=0,&\label{Chen-Str-equ}\\
&\lim_{t\rightarrow 0^+}u_K(\cdot,t)= u_0\quad\mbox{in the weak sense. }
\end{aligned}
\right.
\end{equation}

Moreover, there exists a radius $R=R(\|\nabla u_0\|_{L^2_{loc}},n,m)>0$ and a positive constant $C=C(\|\nabla u_0\|_{L^2_{loc}},n,m)$ such that,
\begin{equation}
\begin{split}
&e_K(u_K)(x,1):=\frac{|\nabla U_K(x)|^2}{2}+\frac{K}{2}\chi\left( d^2_{N}(U_K(x))\right)\leq \frac{C}{|x|^2},\quad |x|\geq R,\\
&\|e_K (u_K)(t)\|_{L^1(B(x_0,1))}\leq C\left(n,m,\|\nabla u_0\|_{L^2_{loc}(\R^n)},t\right)\|\nabla u_0\|^2_{L^2(B(x_0,1))},\quad\forall x_0\in\R^n,\\
&\|\partial_tu_K\|_{L^2((0,t),L^2_{loc}(\R^n))}\leq C(n,m,t)\|\nabla u_0\|_{L^2_{loc}(\R^n)},\label{ene-est-theo-Chen-Str-app}
\end{split}
\end{equation}
where $\lim_{t\rightarrow 0}C\left(n,m,\|\nabla u_0\|_{L^2_{loc}(\R^n)},t\right)=\lim_{t\rightarrow 0}C(n,m,t)=1$.

 Finally, $u_K(t)$ converges strongly to $u_0$ as $t$ goes to $0$ in $H^1_{loc}(\R^n)$.
\end{theo}
 
The proof of Theorem \ref{Chen-Str-app} is divided into four steps explained just before Section \ref{section-well-def-2}. Each step is proved in Sections \ref{section-well-def-2}, \ref{section-CS-well-posed}, \ref{sec-varespilon-reg-thm} and \ref{sec-a-priori-chen-struwe-expanders}.

  In order to prove Theorem \ref{Chen-Str-app}, we reformulate (\ref{Chen-Str-equ}) as a fixed point problem in the perspective of applying the Leray-Schauder Fixed Point Theorem as has been done in \cite{Jia-Sverak} in the context of the Navier-Stokes equation: the main reason why we assume the target Riemannian manifold $(N^{m-1},g)$ is isometrically embedded as a hypersurface of $\R^m$ is to make sense of this fixed point formulation.
 
 In order to guess the qualitative properties of the space of perturbations, we investigate the properties of the first approximations in Section \ref{sec-first-app}.\\

In the case of the target manifold being a round sphere another approximation of the harmonic map flow was introduced by Chen \cite{Chen-Har-Map}. In our setting, it is given by the homogeneous Ginzburg-Landau flow with parameter $K>0$
  \begin{equation}
\left\{\begin{aligned}
&\partial_tu=\Delta u+\frac{K}{t}(1-|u|^2)u,\quad\mbox{on $\mathbb{R}^n\times\mathbb{R}_+$},&\label{eq-GL-K} \\
&u|_{t=0}=u_0.&
\end{aligned}
\right.
\end{equation}
The reason why we introduce the factor $t^{-1}$ in front of the term $K(1-|u|^2)u$ is to make the Ginzburg-Landau flow invariant under the same scaling (\ref{scal-1}) and (\ref{resc-cond}) as the harmonic map flow. Despite its physical relevance, the homogeneous Ginzburg-Landau flow does not seem to give a precise estimate on the singular set of the limiting harmonic map as was noticed by Chen and Struwe. This is essentially due to the lack of a good Bochner formula which in turn is caused by the difficulty of controlling the vanishing set of an expanding solution a priori. 
Still, with the same methods one gets the following result:
\begin{theo}\label{Ginz-Land-app}
Let $u_0:\mathbb{R}^n\rightarrow \mathbb{S}^{m-1}\subset \R^m$ be a $0$-homogeneous map in $C^3_{loc}(\R^n\setminus\{0\})$, $n\geq 3$, such that its restriction to the sphere $\mathbb{S}^{n-1}$ is homotopically trivial. Then for any $K>0$, there exists a smooth Homogeneous Ginzburg-Landau expanding solution $u_K$ coming out of $u_0$:
\begin{equation}
\left\{\begin{aligned}
&\Delta_fU_K+K(1-|U_K|^2)U_K=0,\label{Ginz-Land-equ}&\\
&\lim_{t\rightarrow 0^+}u_K(\cdot,t)= u_0,\quad\mbox{in the weak sense.}&
\end{aligned}
\right.
\end{equation}

Moreover, if $$e_K(u_K)(x,t):=\frac{1}{2}\left(|\nabla u_K|(x,t)^2+\frac{K}{2t}(1-|u_K(x,t)|^2)^2\right),\quad(x,t)\in\R^n\times\R_+,$$ denotes the pointwise energy associated to the solution $u_K$ then:

\begin{eqnarray*}
&&\|e_K (u_K)(t)\|_{L^1(B(x_0,1))}\leq C\left(n,m,\|\nabla u_0\|_{L^2_{loc}(\R^n)},t\right)\|\nabla u_0\|^2_{L^2(B(x_0,1))},\quad\forall x_0\in\R^n,\label{unif-bds-I-GL}\\
&&\|\partial_tu_K\|_{L^2((0,t),L^2_{loc}(\R^n))}\leq C(n,m,t)\|\nabla u_0\|_{L^2_{loc}(\R^n)},\label{unif-bds-II-GL}
\end{eqnarray*}
where $\lim_{t\rightarrow 0}C\left(n,m,\|\nabla u_0\|_{L^2_{loc}(\R^n)},t\right)=\lim_{t\rightarrow 0}C(n,m,t)=1$.

In particular, $u_K(t)$ converges strongly to $u_0$ as $t$ goes to $0$ in $H^1_{loc}(\R^n)$.

\end{theo}

We would like to relate our work to previous articles on this subject. To our knowledge, most of the literature concerns maps from $\R^n$ to an equator of a rotationally symmetric target manifold such as the works of Germain and Rupflin \cite{Ger-Rup}, Biernat and Bizon \cite{Bie-Biz} and the more recent work due to Germain, Ghoul and Miura \cite{Ger-Gho-Miu}. In particular, our setting includes theirs in the case the target is a sphere since a map from $\mathbb{S}^{n-1}$ with values in an equator is homotopic to a constant. Of course, since (\ref{eq-HMP}) reduces to an ODE in such a corotational setting, the above mentioned works obtain more quantitative results even if the question of regularity is not really addressed.

There are at least two other partial differential equations that motivate this work. Jia and \v{S}ver\'ak \cite{Jia-Sverak} proved the existence of smooth expanding solutions of the Navier-Stokes equation. In this case, the homogeneity is of degree $-1$. To prove Theorem \ref{main-theo}, we proceed similarly to their work by using the Leray-Schauder degree theory. For this, one needs a path of initial conditions $(u_0^{\sigma})_{0\leq \sigma\leq 1)}:\mathbb{S}^{n-1}\rightarrow N$ connecting the restriction $u_0^0$ of $u_0$ to $\mathbb{S}^{n-1}$ to a suitable map $u^1_0$. A suitable map $u^1_0$ means here that there is an obvious solution coming out of $u_0^1$ for which there is a uniqueness result for small solutions lying in a suitable function space. In the case of the Navier-Stokes equation, the path is given for free since it suffices to contract the initial vector field to zero. In our case, the path is given by assumption. There is also a deep analogy with the Ricci flow that exhibits the same scale invariance. In the setting of the Ricci flow, $u_0$ is replaced by a metric cone $C(M)$ over a closed Riemannian manifold $(M,g)$ endowed with its Euclidean cone metric $dr^2+r^2g$ and the topological assumption on $M$ similar to the triviality of the homotopy class of $u_0$ is that it is null cobordant. See \cite{Der-Asy-Com-Egs} and \cite{Der-Smo-Pos-Met-Con} in the case $(M,g)$ is a Riemmanian manifold with curvature operator larger than $1$ or \cite{Con-Der} in a more algebraic context. \\

Finally, let us describe the content of each sections.\\

In section \ref{sec-def-theory-HMF} we show a perturbation result for expanding solutions with small energy.

Section \ref{sec-hom-Che-Str} defines the notion of homogeneous Chen-Struwe flow and investigates the existence of smooth expanding solutions to this flow as stated in Theorem \ref{Chen-Str-app}.  Section \ref{sec-first-app} analyses carefully the properties of the first approximation and Section \ref{Section-Fixed-point-formulation} reduces the analysis to a fixed point problem. Sections \ref{section-well-def-2}  and \ref{section-CS-well-posed} prove this fixed point problem is well-posed in the context of Leray-Schauder degree theory. Before proceeding to a priori bounds on expanding solutions to the Homogeneous Chen-Struwe flow, Sections \ref{sec-boch-formula-chen-struwe} and \ref{Energy-inequ-section} establish a Bochner formula and a local entropy monotonicity formula that are crucial to prove an $\varepsilon$-regularity theorem in this setting: see Section  \ref{sec-eps-reg-theo}. Then Section \ref{sec-a-priori-chen-struwe} starts proving a priori $C^0$ bounds of such expanding solutions. Sections \ref{sec-a-priori-weighted-est} and \ref{sec-a-priori-C1-weighted-Che-Str} take care of bounding such expanding solutions at infinity a priori. Section \ref{sec-lip-end-proof} ends the proof of Theorem \ref{Chen-Str-app}. 


Finally, Section \ref{section-Taylor} investigates Taylor expansions at infinity of any expanding solutions of the harmonic map flow in terms of the jets of the initial condition at a formal level. This section is inspired by a similar expansion for asymptotically conical expanding solutions of the Ricci flow done by Lott and Wilson \cite{Lott-Wil}.\\


\textbf{Acknowledgments}: The authors wish to thank the anonymous referees for helpful suggestions that improve the exposition of this paper.

\section{Deformation theory of expanders (with small energy)}\label{sec-def-theory-HMF}

From now on, let $u:\mathbb{R}^n\rightarrow N$ be an expanding solution of the harmonic map flow, fixed once and for all. We consider the linearisation of equation (\ref{eq-HMP}) around $u$. It takes the following time-dependent form: if $u+h$ is an expanding solution then,
\begin{eqnarray*}
\partial_t(u+h)&=&\Delta(u+h)+A(u+h)(\nabla (u+h),\nabla (u+h))|,
\end{eqnarray*}
is equivalent to 
\begin{eqnarray*}
\partial_th&=&\Delta h+ A(u+h)(\nabla (u+h),\nabla (u+h))-A(u)(\nabla u,\nabla u)\\
&=:&\Delta h+R(u,\nabla u,h,\nabla h),\\
\end{eqnarray*}
where
\begin{eqnarray}
R(u,\nabla u,h,\nabla h)&:=&A(u+h)(\nabla (u+h),\nabla (u+h))-A(u)(\nabla u,\nabla u).\label{rema-term}
\end{eqnarray}

Our main result in this section is a uniqueness result for expanding solutions close to a constant map.

First, let us introduce the main function space for the initial conditions:
\begin{eqnarray*}
X_{0}:=\left\{u_0:\mathbb{R}^n\rightarrow N\subset \mathbb{R}^m\quad|\quad  \|u_0\|_{BMO} <+\infty\right\},
\end{eqnarray*}
where 
\begin{eqnarray*}
\| u_0\|_{BMO}:=\sup_{x\in\mathbb{R}^n,r>0}\fint_{B(x,r)}|u_0-\fint_{B(x,r)}u_0|(y)dy.
\end{eqnarray*}
 
 Secondly, the relevant function space $X$ of solutions is defined by:
 
 \begin{eqnarray*}
X:=&&\{u:\mathbb{R}^n\times\R_+\rightarrow \mathbb{R}^m \quad |\quad \|u\|_{L^{\infty}(\mathbb{R}^n\times\R_+)}+\sup_{(x,t)\in \mathbb{R}^n\times\R_+}\sqrt{t}|\nabla u|(x,t)\\
&&+\sup_{(x,s)\in \mathbb{R}^n\times\R_+}\hcancel{\|} \nabla u\|_{L^2(P(x,\sqrt{s}))}<+\infty \},
\end{eqnarray*}
where $P(x,r):=B(x,r)\times(0,r^2)$ and where
\begin{eqnarray*}
\hcancel{\|}h\|^p_{L^p(P(x,r))}:=\fint_{P(x,r)}r^2|h|^p(y,s)dyds,\quad x\in \mathbb{R}^n,\quad r>0, \quad p\geq 1,
\end{eqnarray*}
where $h$ is a tensor defined on $\R^n\times\R_+$. The number $\hcancel{\|}h\|^p_{L^p(P(x,r))}$ is the normalized $p$-norm of $h$ on the parabolic neighborhood $P(x,r)$.

Finally, as in \cite{Koc-Lam-Rou}, we introduce a somewhat intermediate function space $Y$:
\begin{eqnarray*}
Y:=\left\{R:\mathbb{R}^n\rightarrow \mathbb{R}^m\quad|\quad \sup_{t\in\R_+}t\|R(t)\|_{L^{\infty}(\mathbb{R}^n)}+\sup_{s\in\R_+}\nor R\|_{L^1(P(x,\sqrt{s}))}<+\infty\right\}
\end{eqnarray*}

We are now in a position to state the main result of this section:

\begin{theo}
There exists $\varepsilon>0$ such that if $u:\mathbb{R}^n\rightarrow N$ is an expanding solution of the harmonic map flow coming out of $u_0\in X_{0}$ with $\||\nabla u|^2\|_{Y}<\varepsilon$, then there is a neighbourhood $\mathcal{U}_0$ of $u_0$ in the $BMO$-topology such that for any $0$-homogeneous map $v_0\in \mathcal{U}_0$, there exists an expanding solution of the harmonic map flow $v:\mathbb{R}^n\rightarrow N$ coming out of $v_0$. Besides, uniqueness holds in a small ball with center $u$ with respect to the topology of $X$.
\end{theo}

\begin{proof}
The proof follows closely the one in the paper \cite{Wang-BMO} which in turn is motivated by the paper \cite{Koc-Lam-Rou}.

First of all, let us fix some map $v_0\in BMO(\mathbb{R}^n,N)$  and define the map $T: X\rightarrow X$ as follows:
\begin{eqnarray*}
T(h):=k_t\ast (v_0-u_0)+\int_0^tk_{t-s}\ast R(u,\nabla u,h,\nabla h)(s)ds,
\end{eqnarray*}
where $R(u,\nabla u,h,\nabla h)$ is defined by (\ref{rema-term}) and where $k_t$ denotes the Euclidean heat kernel
\begin{eqnarray*}
k_t(x):=(4\pi t)^{-\frac{n}{2}}\exp{\left(-\frac{|x|^2}{4t}\right)},\quad t>0,\quad x\in\mathbb{R}^n.
\end{eqnarray*}

 \begin{claim}
 $T$ is well-defined and the following estimates holds:
 \begin{eqnarray*}
\|T(h)\|_{X}\leq c\left( \|v_0-u_0\|_{BMO}+\left[\||\nabla u|^2\|_{Y}+\|u\|_{X}\|h\|_{X}+\|h\|_{X}^2\right]\|h\|_{X}\right),
\end{eqnarray*}
for some uniform positive constant $c>0$. 
\end{claim}
Indeed, the estimate 
\[
\| k_t\ast (v_0-u_0)\|_{X}\le C \|v_0-u_0\|_{BMO}
\]
can be found at the beginning of section 2 in \cite{Wang-BMO} and it follows from Lemma $3.1$ in \cite{Wang-BMO} that
\[
\|\int_0^tk_{t-s}\ast R(u,\nabla u,h,\nabla h)(s)ds\|_{X} \le C \|R(u,\nabla u,h,\nabla h)\|_{Y}.
\]
Now a standard computation shows that 
\[
\|R(u,\nabla u,h,\nabla h)\|_{Y}\leq C\|u\|_{X}\|h\|_{X}^2+C\|h\|_{X}^3+C\||\nabla u|^2\|_{Y}\|h\|_{X}.
\]
Similarly one observes that 
\begin{claim}
 $T$ is a contraction on a sufficiently small ball around $u$ in $X$, i.e. there is some $q\in(0,1)$ such that 
 \begin{eqnarray*}
 \|T(h_1)-T(h_2)\|_{X}\leq q\|h_1-h_2\|_{X},
 \end{eqnarray*}
 for all $h_1$, $h_2$ in $B_{X}(u, \delta)$ if $\delta$ is small enough.
\end{claim}
Hence we obtain the desired solution $v=u+h$ from the Banach fixed point theorem. It was also shown in \cite{Wang-BMO}, Proof of Theorem $1.3$, that $v$ indeed maps into $N$.
\end{proof}

\section{Preliminaries}\label{sec-hom-Che-Str}
\subsection{Geometry of hypersurfaces of Euclidean space}
Let $(N^{m-1},g)$ be a closed Riemannian manifold which is isometrically embedded in Euclidean space $\R^m$. Let us denote the distance function to $N$ by $d_{N}$. Then there exists a tubular neighbourhood $T_{2\delta_N}(N):=\{y\in\R^m\,|\,d_N(y)<2\delta_N\}$ of $N$ such that the projection map $\Pi_N:T_{2\delta_N}(N)\rightarrow N$ is well-defined and smooth.
As in \cite{Chen-Struwe}, let $\chi$ be a smooth, non-decreasing function such that $\chi(s)=s$ for $0\leq s\leq \delta_N^2$ and $\chi(s)=2\cdot \delta_N^2$ if $s\geq (2\cdot \delta_N)^2$. 

We also recall the definition of the signed distance to $N$. Since $(N^{m-1},g)$ is closed and isometrically embedded in Euclidean space $\R^m$, there is a smooth function $\rho:\R^m\rightarrow \R$ such that 
\begin{equation*}
\begin{split}
&\{x \in\R^m\,|\, \rho(x)<0\}=:\Omega\quad \text{is an open bounded domain with smooth boundary $N$},\\
&\{x\in \R^m\,|\, \rho(x)=0\}=N,\quad\text{and $\nabla\rho(x)\neq 0,\quad x\in\partial\Omega=N$}.
\end{split}
\end{equation*}
Such a function $\rho$ is called a defining function for the domain $\Omega$ which is the "inside" of $N$, the set $\{x\in\R^m\,|\,\rho(x)>0 \}$ being the "outside" of $N$: see [Chapter $1$, Section $1.2$, \cite{Kra-Par}] for a systematic study. The signed distance of $N$ denoted by $\bar{d}_N$ is then defined by:
\begin{eqnarray}
\bar{d}_N(x):=\sgn(\rho(x))d_N(x),\quad x\in\R^m.
\end{eqnarray}
[Theorem $1.2.6$, \cite{Kra-Par}] ensures $\bar{d}_N$ is a smooth function on a tubular neighborhood $T_{2\delta_N}$ of $N$. As a final remark, we notice that $\bar{d}^2_N=d_N^2$.

Let us treat the case of a Euclidean sphere $N=\mathbb{S}^{m-1}\subset\R^m$. A defining function for the unit $m$-Euclidean ball $\mathbb{B}^m(0,1)$ centered at $0\in\R^m$ is $\rho(x)=|x|^2-1$ for $x\in\R^m$. The signed distance of $\mathbb{S}^{m-1}$ is then $\bar{d}_{\mathbb{S}^{m-1}}(x)=|x|-1$ and it is smooth on $\R^m\setminus \{0\}$.

 \subsection{Properties of the first approximation}\label{sec-first-app}

 Denote by $U_0(t)$ the caloric extension of the map $u_0$, i.e. 
 \begin{eqnarray*}
 U_0(x,t):=(k_t\ast u_0)(x),\quad (x,t)\in \R^n\times \R_+,
 \end{eqnarray*}
and denote by $U_0$ the map $U_0(\cdot,1)$. By construction, $U_0$ is $0$-homogeneous and hence
 \begin{eqnarray*}
U_0(x,t)=U_0\left(\frac{x}{\sqrt{t}},1\right)=U_0\left(\frac{x}{\sqrt{t}}\right),\quad \forall(x,t)\in \R^n\times \R_+.
 \end{eqnarray*}

\begin{lemma}\label{lemma-first-app-Chen-Str}
Let $u_0:\R^n\rightarrow(N,g)\subset\R^m$ be a Lipschitz $0$-homogenous map. Then the caloric extension $U_0$ of $u_0$ satisfies
\begin{eqnarray}
&&\|U_0\|_{L^{\infty}}\leq \|u_0\|_{L^{\infty}}, \\
&&\sup_{x\in\R^n}(1+|x|)|\nabla^kU_0|(x)\leq C(k,u_0),\quad \forall k\geq 1,\label{est-cal-1}\\
&&(1+|x|)d_N(U_0(x))\leq(1+|x|)|U_0(x)-u_0(x/|x|)|\leq C(u_0).\label{est-cal-2}
\end{eqnarray}
Moreover, if $u_0$ is in $C^2_{loc}(\R^n\setminus\{0\})$, one has the improved decay
\begin{eqnarray}
&&\sup_{x\in\R^n}(1+|x|^2)^{\frac{\min\{k,2\}}{2}}|\nabla^kU_0|(x)\leq C(u_0),\quad k \geq 1,\label{est-cal-3}\\
&&(1+|x|^2)d_N(U_0(x))\leq(1+|x|^2)|U_0(x)-u_0(x/|x|)|\leq C(u_0).\label{est-cal-4}
\end{eqnarray}
If $u_0$ is in $C^3_{loc}(\R^n\setminus\{0\})$, with $\bm{n\geq 4}$, one has 
\begin{eqnarray}
&&\sup_{x\in\R^n}(1+|x|^3)|\nabla^3U_0|(x)+\sup_{|x|\geq R(u_0)}(1+|x|^3)|\nabla(\bar{d}_N(U_0))|(x)\leq C(u_0),\label{est-cal-5}
\end{eqnarray}
for some positive radius $R(u_0)$.

Moreover, if $(u_0^{\sigma})_{\sigma\in[0,1]}$ is a path of $C^3$ maps such that $u_0^0=u_0$ and $u_0^1\equiv P\in N$ then the constants $(C(u_0^{\sigma}))_{\sigma\in[0,1]}$, $(R(u_0^{\sigma}))_{\sigma\in[0,1]}$ in (\ref{est-cal-2}), (\ref{est-cal-3}), (\ref{est-cal-4}) and (\ref{est-cal-5}) satisfy $$\lim_{\sigma\rightarrow 1}C(u_0^{\sigma})=0,\quad \sup_{\sigma\in[0,1]}R(u_0^{\sigma})<+\infty.$$ 
\end{lemma}

\begin{proof}
The first bound on the $C^0$ norm of $U_0$ follows easily from the maximum principle applied to the heat equation or by using the explicit formula in terms of the Euclidean heat kernel. The decay (\ref{est-cal-1}) on the first derivatives of $U_0$ is proved by using the corresponding decay of the derivatives of the initial condition $u_0$. The bound \eqref{est-cal-2} uses the heat equation in its static form. Since $U_0(t)$ is $0$-homogeneous as a time dependent function then the map $U_0(\cdot,1)=U_0$ satisfies $\Delta_fU_0=0$, i.e.
\begin{eqnarray*}
\frac{r}{2}\partial_rU_0=-\Delta U_0=\textit{O}((1+r)^{-1}),
\end{eqnarray*}
thanks to the previous estimate (\ref{est-cal-1}) on the second derivatives of $U_0$. In particular, this implies that
$$U_0(x)=u_0(x/|x|)+\textit{O}((1+|x|)^{-1}),$$ as $x$ tends to $+\infty$.

Therefore, $d_N(U_0(x))\le |U_0(x)-u_0(x/|x|)|=\textit{O}((1+|x|)^{-1}).$\\

Now, if $u_0$ is in $C^2_{loc}(\R^n\setminus\{0\})$, the decays (\ref{est-cal-3}) on the first and the second derivatives of $U_0$ are proved by using the corresponding decay of the derivatives of the initial data $u_0$. Notice that if $k=0,1, 2$ then $\nabla^k u_0\in L^1_{loc}(\R^n)$, if $n\geq 3$. Let us prove (\ref{est-cal-3}) for the first derivative $\nabla^3U_0$ for instance. Since $u_0$ is not assumed to be in $C^3_{loc}(\R^n\setminus\{0\})$, one derivates twice the map $u_0$ and once the heat kernel $k_1$ at time $t=1$ in the integral representation formula of the caloric extension $U_0$ to get:
\begin{eqnarray*}
|\nabla^3 U_0|(x)&\leq& C_1\int_{\R^n}|x-y|e^{-\frac{|x-y|^2}{4}}\frac{dy}{|y|^2}\leq C_2\int_{\R^n}e^{-\frac{|x-y|^2}{8}}\frac{dy}{|y|^2}=C_2\int_{\R^n}e^{-\frac{|y|^2}{8}}\frac{dy}{|x-y|^2}\\
&\leq&C_2\left(\int_{|y|\leq 2^{-1}|x|}e^{-\frac{|y|^2}{8}}\frac{dy}{|x-y|^2}+\int_{2^{-1}|x|\leq|y|\leq 2|x|}e^{-\frac{|y|^2}{8}}\frac{dy}{|x-y|^2}\right)\\
&&+C_2\int_{|y|\geq 2|x|}e^{-\frac{|y|^2}{8}}\frac{dy}{|x-y|^2}.
\end{eqnarray*}
The first integral on the righthand side of the previous inequality can be estimated as follows for some $x\in\R^n\setminus\{0\}$:
\begin{eqnarray*}
\int_{|y|\leq 2^{-1}|x|}e^{-\frac{|y|^2}{8}}\frac{dy}{|x-y|^2}\leq\frac{4}{|x|^2} \int_{|y|\leq 2^{-1}|x|}e^{-\frac{|y|^2}{8}}dy\leq \frac{C}{|x|^2},
\end{eqnarray*}
by the triangular inequality $|x-y|\geq |x|-|y|\geq 2^{-1}|x|$ if $y\in B(0,2^{-1}|x|)$. 
 
 The second one is handled as follows:
 \begin{eqnarray*}
\int_{2^{-1}|x|\leq|y|\leq 2|x|}e^{-\frac{|y|^2}{8}}\frac{dy}{|x-y|^2}&\leq& e^{-\frac{|x|^2}{32}}\int_{2^{-1}|x|\leq|y|\leq 2|x|}\frac{dy}{|x-y|^2}\leq e^{-\frac{|x|^2}{32}}\int_{|x-y|\leq 3|x|}\frac{dy}{|x-y|^2}\\
&\leq&C|x|^{n-2}e^{-\frac{|x|^2}{32}},
\end{eqnarray*}
where $C$ is a positive constant uniform in $x\in\R^n$ since $|x|^{-2}\in L^1_{loc}(\R^n)$.

Finally, the third integral can be estimated similarly.

We proceed similarly to (\ref{est-cal-2}) to prove (\ref{est-cal-4}), since $U_0(t)$ is $0$-homogeneous as a time dependent function then the map $U_0$ satisfies $\Delta_fU_0=0$, i.e.
\begin{eqnarray}
\frac{r}{2}\partial_rU_0=-\Delta U_0=\textit{O}((1+r^2)^{-1}),\label{evo-eqn-U_0-static}
\end{eqnarray}
thanks to the previous estimate on the second derivatives of $U_0$. In particular, this implies that
$$U_0(x)=u_0(x/|x|)+\textit{O}((1+|x|^2)^{-1}),$$ as $x$ tends to $+\infty$ and therefore $\bar{d}_N(U_0(x))=\textit{O}((1+|x|^2)^{-1})$ as $x$ tends to $+\infty$.

It remains to prove (\ref{est-cal-5}). The first estimate on the third derivatives can be proved as we proceeded for (\ref{est-cal-3}). Let us handle the second term on the lefthand side of (\ref{est-cal-5}).

 First of all, since the function $\bar{d}_N$ is differentiable on the set $T_{2\delta_N}(N)$, estimate (\ref{est-cal-4}) ensures that $d_N(U_0)<2\delta_N$ as soon as $|x|\geq R(u_0)$. Now observe that:
\begin{eqnarray*}
\left|\nabla \bar{d}_N(U_0(x))-\nabla\bar{d}_N(u_0(x/|x|))\right|\leq C|U_0(x)-u_0(x/|x|)|\leq \frac{C(u_0)}{1+|x|^2},
\end{eqnarray*}
by (\ref{est-cal-4}).
Therefore we get the following intermediate estimate for $x\in\R^n\setminus\{0\}$:
\begin{eqnarray*}
|\nabla(\bar{d}_N(U_0))|(x)&=&|\nabla\bar{d}_N(U_0)(\nabla U_0)|(x)\\
&\leq&|\nabla\bar{d}_N(u_0)(\nabla U_0)|(x)+\frac{C(u_0)}{1+|x|^2}|\nabla U_0|\\
&\leq&|\nabla\bar{d}_N(u_0)(\nabla U_0)|(x)+\frac{C(u_0)}{1+|x|^3}\\
&\leq&|\nabla\bar{d}_N(u_0)(\nabla u_0)|(x)+C(u_0)|\nabla( U_0-u_0)|(x)+\frac{C(u_0)}{1+|x|^3}\\
&=&C(u_0)|\nabla( U_0-u_0)|(x)+\frac{C(u_0)}{1+|x|^3},
\end{eqnarray*}
where we use (\ref{est-cal-3}) with $k=1$ in the third line and $\bar{d}_N(u_0)=0$ in the last line. To conclude, it suffices to show that $\nabla(U_0-u_0)(x)=\textit{O}((1+|x|^3)^{-1}).$

Let us remark that the radial derivative of $U_0$ hence the radial derivative of $U_0-u_0$ is decaying as expected by (\ref{evo-eqn-U_0-static}). Now differentiate (\ref{evo-eqn-U_0-static}) once and observe that the hessian of $|x|^2$ is $2\delta$ where $\delta$ denotes Euclidean metric. One gets:
\begin{eqnarray*}
-\Delta\nabla U_0&=&\frac{r}{2}\partial_rU_0+\frac{1}{2}\nabla U_0\\
&=&\frac{1}{2}\partial_r\left(r\nabla U_0\right).
\end{eqnarray*}
Estimate (\ref{est-cal-5}) on the third derivatives then implies that $$\partial_r\left(r\nabla U_0\right)=\textit{O}((1+|x|^3)^{-1}).$$
Integrating the previous differential equality along a ray gives:
\begin{eqnarray*}
r\nabla U_0(r,\omega)-\nabla^{sph} u_0(\omega)=\textit{O}((1+r)^{-2}),\quad (r,\omega)\in\R_+\times\mathbb{S}^{n-1},
\end{eqnarray*}
where $\nabla^{sph} u_0$ denotes the spherical derivatives of $u_0$. This fact leads to the expected estimate.
\end{proof}

The previous Lemma \ref{lemma-first-app-Chen-Str} and the analysis to follow in the next sections show that if $n\geq 4$, the choice of the caloric extension of $u_0$ as a first approximation suffices. If $n=3$, we consider a barycentric approximation of $u_0$ as follows: if $\eta:\R^n\rightarrow[0,1]$ denotes any smooth function such that $\eta\equiv 1$ if $|x|\geq 2$ and $\eta\equiv 0$ if $|x|\leq 1$, we define 
\begin{eqnarray}
U_0^b:=(1-\eta)P+\eta u_0,\label{sec-bar-app}
\end{eqnarray}
where $P\in N$ is fixed.

The properties of $U^b_0$ can be summarized as follows:
\begin{lemma}\label{lemma-bar-app}
With the previous notations, assume that $u_0$ is in $C^3_{loc}(\R^n\setminus\{0\})$, then the barycentric approximation satisfies
\begin{eqnarray}
&&\|U^b_0\|_{L^{\infty}}\leq |P|+\|u_0\|_{L^{\infty}},\quad \forall x\in\R^n,\quad U^b_0=u_0,\quad |x|\geq 2,\label{est-bar-app-1}\\
&&\sup_{x\in\R^n}(1+|x|^{k})|\nabla^kU^b_0|(x)\leq C(k,u_0)<+\infty,\quad 0\leq k\leq 3,\label{est-bar-app-2}\\
&&\sup_{x\in\R^n}(1+|x|^{2})|\Delta_fU^b_0|\leq C(u_0)<+\infty,\label{est-bar-app-3}\\
&&\sup_{x\in\R^n}(1+|x|^2)d_N(U_0^b(x))\leq C(u_0),\label{est-bar-app-4}
\end{eqnarray}
where $|P|$ denotes the euclidean norm of $P\in\R^m$. In particular, $\bar{d}_N(U_0^b)(x)=0$ if $|x|\geq 2$.

Moreover, if $(u_0^{\sigma})_{\sigma\in[0,1]}$ is a path of $C^3$ maps such that $u_0^0=u_0$ and $u_0^1\equiv P\in N$ then the constants $(C(u_0^{\sigma}))_{\sigma\in[0,1]}$ in (\ref{est-bar-app-2}), (\ref{est-bar-app-3}) and (\ref{est-bar-app-4}) satisfy $\lim_{\sigma\rightarrow 1}C(u_0^{\sigma})=0.$ 

\end{lemma}

\begin{proof}
By definition of the map $U_0^b$, $U_0^b\equiv u_0$ outside $B(0,2)\subset\R^n$. Moreover, by the triangular inequality: $|U_0^b|(x)\leq (1-\eta(x))|P|+\eta(x)\|u_0\|_{L^{\infty}}$ Since $\eta$ takes its values into $[0,1]$, (\ref{est-bar-app-1}) follows. The estimate (\ref{est-bar-app-2}) comes from the fact that $\nabla \eta$ is compactly supported in $B(0,2)$ and the decay of the derivatives of $u_0$ outside the origin $0\in\R^n$. 

Now since $u_0$ is $0$-homogeneous, $\Delta_fU^b_0=\Delta_f u_0=\Delta u_0$ on $\R^n\setminus B(0,3)$ which decays quadratically. Therefore, (\ref{est-bar-app-3}) follows immediately. 

Similarly to the proof of (\ref{est-cal-4}), one observes that:
\begin{eqnarray*}
(1+|x|^2)d_N(U_0^b(x))&\leq& (1+|x|^2)|U_0^b(x)-u_0(x/|x|)|\\
&\leq&(1+|x|^2)(1-\eta(x))|P-u_0(x/|x|)|,
\end{eqnarray*}
for any $x\in \R^n$. In particular, if $|x|\geq 2$ then the righthand side of the previous inequality vanishes. Therefore,
\begin{eqnarray*}
(1+|x|^2)d_N(U_0^b(x))&\leq&5\sup_{|x|\leq 2}|P-u_0(x/|x|)|,
\end{eqnarray*}
which implies (\ref{est-bar-app-4}).
\end{proof}

Because of the similarities shared by Lemma \ref{lemma-first-app-Chen-Str} with Lemma \ref{lemma-bar-app}, we will only consider the ansatz with the caloric extension in the next sections. 

\subsection{Fixed point formulation}\label{Section-Fixed-point-formulation}

Motivated by Lemmata \ref{lemma-first-app-Chen-Str} and \ref{lemma-bar-app}, we introduce the Banach space
 \begin{eqnarray*}
X:=\left\{V\in C^1_{loc}(\R^n,\R^m)\quad|\quad \sup_{x\in\R^n}(1+|x|)^{2}|V(x)|+(1+|x|)^{3}|\nabla V(x)|<+\infty\right\}.
\end{eqnarray*}
In order to produce a solution to (\ref{Chen-Str-equ}), we look for a solution $u_K$ to the homogeneous Chen-Struwe equation of the form
\begin{eqnarray}
\partial_tu_K&=&\Delta u_K-\frac{K}{t}\chi'\left( d^2_{N}(u_K)\right)\nabla\left( \frac{d^2_{N}}{2}\right)(u_K),\quad t>0,\label{Chen-Str-equ-non-stat}\\
u_K(x,t)&:=&U_0\left(\frac{x}{\sqrt{t}}\right)+V_K\left(\frac{x}{\sqrt{t}}\right),
\end{eqnarray}
where $U_0:\R^n\rightarrow \R^m$ denotes the caloric extension of $u_0$ as defined in Section \ref{sec-first-app} and where $V_K\in X$. Notice that the time-dependent map $V_K(\cdot,t):=V_K(\cdot/\sqrt{t})$ will satisfy the bounds
\begin{eqnarray}
(\sqrt{t}+|x|)^{2}|V_K(x,t)|+(\sqrt{t}+|x|)^{3}|\nabla V_K(x,t)|\leq Ct,\quad \forall x\in\R^n,\quad \forall t>0.\label{time-bd-error}
\end{eqnarray}

As we want to use the Leray-Schauder Fixed Point Theorem in order to show Theorem \ref{Chen-Str-app}, we need to reformulate this problem as follows. Let $\sigma\in[0,1]$ be a parameter and denote by $(u_0^{\sigma})_{\sigma\in[0,1]}$ a path of $C^3$ maps such that $u_0^0=u_0$ and $u_0^1\equiv P\in N$. Note that this path is chosen inside the homotopy class of $[u_0]\in\pi_{n-1}(N)$. \\

Thus, solving (\ref{Chen-Str-equ}) amounts to solving the static Chen-Struwe equation
\begin{eqnarray}
&&\Delta_fV_K-K\chi'\left( d^2_{N}(U_0^{\sigma}+V_K)\right)\nabla\left( \frac{d^2_N}{2}\right)(U_0^{\sigma}+V_K)=0\label{Chen-Str-equ-stat},\quad V_K\in X.
\end{eqnarray}

If $V\in X$, $K>0$ and $\sigma\in[0,1]$, we denote formally by $F_K^{\sigma}(V)\in X$ the solution to the problem

\begin{eqnarray*}
\Delta_fF_K^{\sigma}(V)-K\chi'(d^2_{N}(U_0^{\sigma}))d_{U_0^{\sigma}}\bar{d}_N(F_K^{\sigma}(V))\nabla \bar{d}_N(U_0^{\sigma})&=&-K\chi'(d^2_{N}(U_0^{\sigma}))d_{U_0^{\sigma}}\bar{d}_N(V)\nabla \bar{d}_N(U_0^{\sigma})\\
&&+K\chi'\left( d^2_{N}(U_0^{\sigma}+V)\right)\nabla\left( \frac{d^2_{N}}{2}\right)(U_0^{\sigma}+V),\label{Chen-Str-equ-expander}
\end{eqnarray*}
where $d_{p}\bar{d}_N(v)=<\nabla \bar{d}_N(p),v>$ denotes the differential (when it makes sense) of the signed distance function $\bar{d}_N$ at the point $p\in\R^m$ evaluated at the vector $v\in T_p\R^m$.

\begin{rk}\label{rk-cpct-op}
The reason why we isolate the term $d_{U_0^{\sigma}}\bar{d}_N(F_K^{\sigma}(V))\nabla \bar{d}_N(U_0^{\sigma})$ from the right hand side is that the map $$V\in X\rightarrow d_{U_0^{\sigma}}\bar{d}_N(F_K^{\sigma}(V))\nabla \bar{d}_N(U_0^{\sigma})\in X,$$ is not a compact operator.
\end{rk}

We proceed in three steps:
\begin{enumerate}
\item The map $F_K:X\times[0,1]\rightarrow X$ is a well-defined compact continuous map: this is the content of Section \ref{section-well-def-2}.\\
\item The Leray-Schauder degree of $I-F_K^{\sigma}:B_{X}(0,\varepsilon)\rightarrow B_{X}(0,\varepsilon)$ is $1$ when $\sigma$ is close to $1$, for some positive $\varepsilon$: this is proved in Section \ref{section-CS-well-posed}.\\
\item (A priori estimates) There is a positive constant $M$ (uniform in $\sigma\in[0,1]$) such that if $V\in X$ is such that $F_K^{\sigma}(V)=V$ then $\|V\|_{X}\leq M$: these are the contents of Sections \ref{sec-varespilon-reg-thm} and \ref{sec-a-priori-chen-struwe-expanders}.
\end{enumerate}

\subsection{$F_K$ is a well-defined compact and continuous map}\label{section-well-def-2}
In this section, we prove that the map $F_K^{\sigma}:X\to X$ is compact and continuous.

Note that the map $F_K^{\sigma}$ can also be formally interpreted as 
\begin{eqnarray*}
F_K^{\sigma}(V)(x,t)=-K\int_0^t\int_{\R^n}\mathcal{K}^{\sigma}_{t-s}(x,y)\frac{1}{s}\chi'(d^2_{N}(U_0^{\sigma}))d_{U_0^{\sigma}}\bar{d}_N(V)\nabla \bar{d}_N(U_0^{\sigma})dyds\\
+K\int_0^t\int_{\R^n}\mathcal{K}^{\sigma}_{t-s}(x,y)\frac{1}{s}\chi'\left( d^2_{N}(U_0^{\sigma}+V)\right)\nabla\left( \frac{d^2_{N}}{2}\right)(U_0^{\sigma}+V))(y,s)dyds,
\end{eqnarray*}
where $V(x,t):=V(x/\sqrt{t})$ with $V\in X$ and where $\mathcal{K}^{\sigma}_{t}\in\mathcal{L}(X,X)$ denotes the solution of
\begin{eqnarray*}
\partial_t\mathcal{K}^{\sigma}_{t}&=&\Delta \mathcal{K}^{\sigma}_{t}-\frac{K}{t}\chi'(d^2_{N}(U_0^{\sigma}))d_{U_0^{\sigma}}\bar{d}_N(\mathcal{K}^{\sigma}_t)\nabla \bar{d}_N(U_0^{\sigma}),\\
\lim_{t\rightarrow 0^+}\mathcal{K}^{\sigma}_{t}&=&\delta_0.
\end{eqnarray*}
The issue here is to make sense of this solution and therefore, we prefer to work with the static equation only. To do so, given $V\in X$, we first solve the following Dirichlet problem 

\begin{equation}
\begin{split}
 \Delta_fW_R-K\chi'(d^2_{N}(U_0^{\sigma}))d_{U_0^{\sigma}}\bar{d}_N(W_R)\nabla \bar{d}_N(U_0^{\sigma})&=Q(U_0^{\sigma},V),\quad\text{ on } B(0,R)\subset \R^n,\\
 W_R&=0,\quad  \text{ on } \partial B(0,R),
 \end{split}
\end{equation}
where
\begin{equation}
\begin{split}
Q(U_0^{\sigma},V)&:=-K\chi'(d^2_{N}(U_0^{\sigma}))d_{U_0^{\sigma}}\bar{d}_N(V)\nabla \bar{d}_N(U_0^{\sigma})\label{def-Q-1}\\
&+K\chi'\left( d^2_{N}(U_0^{\sigma}+V)\right)\nabla\left( \frac{d^2_{N}}{2}\right)(U_0^{\sigma}+V).
\end{split}
\end{equation}
 One can prove that such a solution $W_R$ exists and is unique by the maximum principle. 
 
We start with a lemma analysing the behaviour of the righthand side $Q(U_0^{\sigma},V)$ for $V\in X$.
\begin{lemma}\label{est-Q-lemma}
Let $\sigma\in[0,1]$. Then
\begin{equation}
\begin{split}
\|Q(U_0^{\sigma},V)\|_X&\leq C(u_0^{\sigma},K)(1+\|V\|_X)+C(K)\|V\|_X^2,\quad V\in B_X(0,1),\label{est-Q-1}\\
\|Q(U_0^{\sigma},V_2)-Q(U_0^{\sigma},V_1)\|_X&\leq C(u_0^{\sigma},K)\|V_2-V_1\|_X+C(K)\|V_1+V_2\|_X\|V_2-V_1\|_X,  
\end{split}
\end{equation}
where $V_1,V_2\in B_X(0,1)$ and where the constant $C(u_0^{\sigma},K)$ satisfies $\lim_{\sigma\rightarrow 1}C(u_0^{\sigma},K)=0$ when $K$ is fixed.
\end{lemma} 

\begin{proof}
Since $K$ is fixed and the expected estimates are depending on $K$ a priori, we can assume $K=1$ to lighten the notations. 

A Taylor expansion of degree $2$ around $U_0^{\sigma}$ of the righthand side of (\ref{def-Q-1}) gives:
\begin{equation}
\begin{split}
\chi'\left( d^2_{N}(U_0^{\sigma}+V)\right)\nabla\left( \frac{d^2_{N}}{2}\right)(U_0^{\sigma}+V)&=\chi'(d_N^2(U_0^{\sigma}))\nabla\left(\frac{d_N^2}{2}\right)(U_0^{\sigma})\\
&+\chi''(d_N^2(U_0^{\sigma}))d_{U_0^{\sigma}}d_N^2(V)\nabla\left(\frac{d_N^2}{2}\right)(U_0^{\sigma})\\
&+\chi'(d_N^2(U_0^{\sigma}))\nabla^2_{U_0^{\sigma}}\left(\frac{d_N^2}{2}\right)(V)+V\ast V.\label{tay-exp-Q-data}
\end{split}
\end{equation}
where, if $A$ and $B$ are two tensors, $A\ast B$ denotes any contraction of linear combinations of the tensor product $A\otimes B$.
Therefore, by using that 
\begin{eqnarray*}
\chi'(d_N^2(U_0^{\sigma}))\nabla^2\left(\frac{d_N^2}{2}\right)(V)=\chi'(d_N^2(U_0^{\sigma}))\left(d_{U_0^{\sigma}}\bar{d}_N(V)\nabla\bar{d}_N(U_0^{\sigma})+\bar{d}_N(U_0^{\sigma})\nabla^2_{U_0^{\sigma}}\bar{d}_N(V)\right),
\end{eqnarray*}
one gets,
\begin{eqnarray}
Q(U_0^{\sigma},V)=\chi'(d_N^2(U_0^{\sigma}))\bar{d}_N(U_0^{\sigma})\nabla \bar{d}_N(U_0^{\sigma})+\bar{d}_N(U_0^{\sigma})\textit{O}(V)+V\ast V, \label{est-Q-inf}
\end{eqnarray}
 By using Lemma \ref{lemma-first-app-Chen-Str} together with (\ref{est-Q-inf}), one gets:
\begin{eqnarray*}
\sup_{x\in\R^n}(1+|x|^2)|Q(U_0^{\sigma},V)|(x)\leq C(u_0^{\sigma})\left(1+\sup_{x\in\R^n}(1+|x|^2)|V|(x)\right)+C\left(\sup_{x\in\R^n}(1+|x|^2)|V|(x)\right)^2,
\end{eqnarray*}
if $V\in B_X(0,1)$. By differentiating (\ref{est-Q-inf}) and by using Lemma \ref{lemma-first-app-Chen-Str} once more, one gets the first line of (\ref{est-Q-1}). The estimate on the second line of (\ref{est-Q-1}) can be proved similarly.
\end{proof}
In order to obtain a solution defined on all of $\R^n$, we first establish an a priori $C^0$ bound.

\begin{prop}\label{propa-priori-weighted-diri}
The solution $W_R$ defined above satisfies the following a priori weighted $C^0$ bound:
\begin{eqnarray}
\max_{B(0,R)}f|W_R|\leq C(n,m,K)\|Q(U_0^{\sigma},V)\|_{X}.\label{a-priori-weighted-diri}
\end{eqnarray}
In particular, $$\|W_R\|_{C^{2,\beta}(B(0,R_0))}\leq C(\beta,n,m,R_0,K)\|Q(U_0^{\sigma},V)\|_{X},\quad \forall R_0<R,\quad\beta\in(0,1).$$
\end{prop}

\begin{proof}
Since $Q:=Q(U_0^{\sigma},V)$ is $C^1$, the last assertion on the derivatives of $W_R$ follows immediately by interior elliptic Schauder estimates in case the a priori $C^0$ bound holds. Therefore, it suffices to establish (\ref{a-priori-weighted-diri}). For this we note that
\begin{eqnarray*}
\Delta_f|W_R|^2&\geq& 2|\nabla W_R|^2+2K\chi'(d^2_{N}(U_0^{\sigma}))\left(d_{U_0^{\sigma}}\bar{d}_N(W_R)\right)^2-2|Q||W_R|\\
&\geq&2|\nabla W_R|^2-2|Q||W_R|.
\end{eqnarray*}
In particular, for $\varepsilon>0$ we consider the function $W_R^{\varepsilon}:=\sqrt{|W_R|^2+\varepsilon^2}$ and we get that $W_R^{\varepsilon}$ satisfies
\begin{eqnarray*}
\Delta_f|W_R^{\varepsilon}|\geq-|Q|\geq -\|Q\|_Xf^{-1},
\end{eqnarray*}
since $Q\in X$. Now observe that the function $f^{-1}$ is a good barrier function since
\begin{eqnarray*}
\Delta_ff^{-1}&=&-f^{-2}\Delta_ff+2|\nabla f|^2f^{-3}\\
&=&-f^{-1}\left(1-2f^{-2}\frac{|x|^2}{4}\right)\\
&\leq&-Cf^{-1},
\end{eqnarray*}
for some positive constant $C$. Indeed, one can check that for $n\geq 2$ we have
$$\inf_{x\in\R^n} \left(1-2\frac{|x|^2}{4f^2(x)}\right)>0.$$

Hence, for some sufficiently large constant $A$ depending linearly on $\|Q\|_{X}$ and which is independent of $\varepsilon$,
\begin{eqnarray*}
\Delta_f\left(|W_R^{\varepsilon}|-Af^{-1}\right)>0,\quad\text{on}\quad B(0,R).
\end{eqnarray*}
The maximum principle forces the function $|W_R^{\varepsilon}|-Af^{-1}$ to attain its maximum at the boundary $\partial B(0,R)$, hence
\begin{eqnarray*}
\max_{B(0,R)}|W_R^{\varepsilon}|-Af^{-1}&=&\max_{\partial B(0,R)}|W_R^{\varepsilon}|-Af^{-1}\\
&\leq&\varepsilon.
\end{eqnarray*}
The result follows by letting $\varepsilon$ go to $0$. 
\end{proof}

By Proposition \ref{propa-priori-weighted-diri}, one can extract a subsequence $(W_{R_k})_{k\geq 0}$ for a sequence of radii $(R_k)_{k\geq 0}$ going to $+\infty$, that converges in the $C^{2,\beta}_{loc}$ topology, for any $\beta\in(0,1)$, to a vector field $W:\R^n\rightarrow\R^m$ satisfying:
\begin{eqnarray}
 \Delta_fW-K\chi'(d^2_{N}(U_0^{\sigma}))d_{U_0^{\sigma}}\bar{d}_N(W)\nabla \bar{d}_N(U_0^{\sigma})&=&Q(U_0^{\sigma},V),\quad\text{ on } \R^n,\\
 \sup_{\R^n}f|W|&\leq& C(n,m,K)\|Q(U_0^{\sigma},V)\|_{X}.\label{a-priori-bd-F}
\end{eqnarray}
This solution $W$ is unique among regular solutions that converge to $0$ at infinity by the maximum principle. Therefore, the map $F_K^{\sigma}$ makes sense provided the gradient of $F_K^{\sigma}(V)$ decays like $f^{-3/2}$, i.e. $F_K^{\sigma}(V)\in X$. This is the content of the next proposition.

\begin{prop}\label{a-priori-c1-bd-sol-F}
The solution $F_K^{\sigma}(V)$ satisfies the weighted a priori $C^1$ bound
\begin{eqnarray}
\sup_{\R^n}f^{3/2}|\nabla F_K^{\sigma}(V)|\leq C(n,m,K)\|Q(U_0^{\sigma},V)\|_{X}.
\end{eqnarray}

\end{prop}
\begin{proof}
Strictly speaking, the solution $F_K^{\sigma}(V)=:F(V)$ is in $C^{2,\beta}_{loc}$ only. Since the bound we want to establish only involves the first derivatives of $F(V)$ and of $Q(U_0^{\sigma},V)$, we can assume $V$ and hence $F(V)$ and $Q(U_0^{\sigma},V)$ are smooth. Moreover, by interior parabolic Schauder estimates applied to the corresponding time-dependent solution $F(V)(x,t):=F(V)(x/\sqrt{t})$ defined on $\R^n\times \R_+$, one gets the bounds
\begin{eqnarray}
\sup_{x\in\R^n}f(x)\|F(V)\|_{C^{2,\beta}(B(x,1))}\leq C(n,m,\beta,K)\|Q(U_0^{\sigma},V)\|_{X}.\label{rough-bds-F-sol}
\end{eqnarray}

We compute the evolution equation (in disguise) of the gradient $\nabla F(V)$:
\begin{eqnarray*}
\Delta_f\nabla F(V)&=&\nabla(\Delta_fF(V))-\frac{1}{2}\nabla F(V)\\
&=&-\frac{1}{2}\nabla F(V)+K\nabla\left(\chi'(d^2_{N}(U_0^{\sigma}))d_{U_0^{\sigma}}\bar{d}_N(F(V))\nabla \bar{d}_N(U_0^{\sigma})\right)+\nabla \left(Q(U_0^{\sigma},V)\right),
\end{eqnarray*}
which implies, by (\ref{a-priori-bd-F}) and the fact that $Q(U_0^{\sigma},V)\in X$,
\begin{equation}
\begin{split}
\Delta_f|\nabla F(V)|^2&\geq 2|\nabla^2F(V)|^2 -|\nabla F(V)|^2+2K\chi'(d^2_{N}(U_0^{\sigma}))\left[d_{U_0^{\sigma}}\bar{d}_N(\nabla F(V))\right]^2\\
&-C\|Q(U_0^{\sigma},V)\|_{X}f^{-3/2}|\nabla F(V)|\\
&\geq2|\nabla^2F(V)|^2 -|\nabla F(V)|^2-C\|Q(U_0^{\sigma},V)\|_{X}f^{-3/2}|\nabla F(V)|.\label{evo-equ-F-V-grad}
\end{split}
\end{equation}
Indeed,
\begin{eqnarray*}
&&\left\langle\nabla\left(\chi'(d^2_{N}(U_0^{\sigma}))d_{U_0^{\sigma}}\bar{d}_N(F(V))\nabla \bar{d}_N(U_0^{\sigma})\right),\nabla F(V)\right\rangle=\\
&&\chi''(d^2_{N}(U_0^{\sigma}))\langle\nabla\bar{d}_N(U_0^{\sigma}),F(V)\rangle\cdot \langle\nabla\bar{d}_N(U_0^{\sigma}),\nabla_{\nabla(d_N^2(U_0^{\sigma}))} F(V)\rangle\\
&&+\chi'(d^2_{N}(U_0^{\sigma}))\left(\nabla U_0^{\sigma}\ast F(V)+\left[d_{U_0^{\sigma}}\bar{d}_N(F(V))\right]^2\right).
\end{eqnarray*}
Therefore, (\ref{a-priori-bd-F}) and [(\ref{est-cal-5}), Lemma \ref{lemma-first-app-Chen-Str}] implies the estimate (\ref{evo-equ-F-V-grad}).

Now, recall that $\Delta_ff^2=2f^2+2|\nabla f|^2\geq 2f^2$ and multiply the differential inequality (\ref{evo-equ-F-V-grad}) by $f^{2}$ to absorb the term $-|\nabla F(V)|^2$ as follows
\begin{eqnarray*}
\Delta_f\left(f^{2}|\nabla F(V)|^2\right)&\geq&2f^{2}|\nabla^2F(V)|^2-8f^{3/2}|\nabla^2 F(V)||\nabla F(V)|\\
&&+ f^{2}|\nabla F(V)|^2-C(K,\|Q(U_0^{\sigma},V)\|_X)f^{-1/2}\left(f|\nabla F(V)|\right)\\
&\geq&\left(1-Cf^{-1}\right)f^{2}|\nabla F(V)|^2-C\|Q(U_0^{\sigma},V)\|_Xf^{-1}\\
&\geq&f^{2}|\nabla F(V)|^2-C\|Q(U_0^{\sigma},V)\|_Xf^{-1},
\end{eqnarray*}
where we used Young inequality together with (\ref{rough-bds-F-sol}) to get the last inequality and $C$ denotes a positive constant independent of $V$ and $F(V)$ that may vary from line to line.\\


Observe that $\Delta_f\ln f=1-|\nabla\ln f|^2\leq 1$ on $\R^n$.\\

By considering a function of the form $k^{-1}\ln f$ where $k$ is a positive integer, one gets:
\begin{eqnarray*}
\Delta_f\left(f^{2}|\nabla F(V)|^2-k^{-1}\ln f-Af^{-1}\right)&\geq&f^{2}|\nabla F(V)|^2-C\|Q(U_0^{\sigma},V)\|_Xf^{-1}-k^{-1}+ACf^{-1}\\
&\geq&f^2|\nabla F(V)|^2-k^{-1}\ln f-Af^{-1}-k^{-1},
\end{eqnarray*}
for any positive constant $A$ larger than $C\|Q(U_0^{\sigma},V)\|_X$, for some positive constant $C$ uniform in $\sigma\in[0,1]$.
Now, as we know that the function $f^{2}|\nabla F(V)|^2$ is bounded, the function $f^{2}|\nabla F(V)|^2-k^{-1}\ln f-Af^{-1}$ goes to $-\infty$ as $x$ goes to $+\infty$. Therefore, it attains its maximum. The maximum principle applied to the previous differential inequality implies
$$f^{2}|\nabla F(V)|^2-k^{-1}\ln f-Af^{-1}\leq k^{-1},\quad\text{on $\R^n$.}$$
As $A$ is independent of $k$, we can let $k$ go to $+\infty$ and get the expected result.
\end{proof}

We now claim that the map $F_K$ is a compact operator, the proof of its continuity being analogous:

\begin{prop}\label{map-cont-cpct}
The map $F_K: X\times[0,1]\rightarrow X$ is a continuous and compact map.
\end{prop}

\begin{proof}
We only prove the compactness of the map $F_K^{\sigma}$ where $\sigma\in[0,1]$. Let $(V_i)_{i\geq 0}$ be a bounded sequence in $X$. According to (\ref{rough-bds-F-sol}), the sequence $(F_K^{\sigma}(V_i))_i$ subconverges in the $C^{2,\beta}_{loc}$ topology to a map that belongs to $X$. In order to prove that the convergence holds in $X$, it is sufficient to prove that for any $V\in X$ and $i=0,1$ the following estimates hold:
\begin{eqnarray*}
\sup_{\R^n}f^{2+i/2}|\nabla^i(F_K^{\sigma}(V)-F_K^{\sigma}(0))|\leq C(K,n,m,\|V\|_{X}).
\end{eqnarray*}

In the proof of these estimates, for the sake of clarity, we omit the dependence of $F_K^{\sigma}$ and $U_0^{\sigma}$ on $K$ and $\sigma$.

Recall that $F_K^{\sigma}(0)=:F(0)$ is the unique solution in $X$ of: 
\begin{eqnarray*}
\Delta_fF(0)-K\chi'(d^2_{N}(U_0))d_{U_0}\bar{d}_N(F(0))\nabla \bar{d}_N(U_0)&=&Q(U_0,0)\\
&=&K\chi'(d^2_{N}(U_0))\nabla\left(\frac{d_N^2}{2}\right)(U_0).
\end{eqnarray*}

Therefore, $G(V):=F(V)-F(0)\in X$ is a solution of
\begin{eqnarray*}
\Delta_fG(V)-K\chi'(d^2_{N}(U_0^{\sigma}))d_{U_0^{\sigma}}\bar{d}_N(G(V))\nabla \bar{d}_N(U_0^{\sigma})&=&Q(U_0,V)-Q(U_0,0).
\end{eqnarray*}
Now, a Taylor expansion of degree $2$ as in (\ref{tay-exp-Q-data}), using the very definition (\ref{def-Q-1}) of $Q(U_0,V)$, shows that the differential of $Q$ with respect to the variable $V$ satisfies:
\begin{equation}
\begin{split}
D_2Q(U_0,0)(V)&=K\chi'(d^2_{N}(U_0^{\sigma}))\bar{d}_N(U_0)\nabla^2\bar{d}_N(U_0)(V)\\
&+K\chi''(d_N^2(U_0^{\sigma}))d_{U_0^{\sigma}}d_N^2(V)\nabla\left(\frac{d_N^2}{2}\right)(U_0^{\sigma}).\label{first-der-Q-sec-var}
\end{split}
\end{equation}
Notice that the second term on the righthand side of (\ref{first-der-Q-sec-var}) is compactly supported in a ball of $\R^n$ whose radius is independent of $\sigma\in[0,1]$ by the definition of the function $\chi$ and Lemma \ref{lemma-first-app-Chen-Str}.

In particular, by Taylor's theorem together with the estimates of Lemma \ref{lemma-first-app-Chen-Str}, 
\begin{eqnarray*}
\sup_{\R^n}f^{2+i/2}|\nabla^i(Q(U_0,V)-Q(U_0,0))|\leq C(K,n,m,\|V\|_{X}),\quad i=0,1.
\end{eqnarray*}
As in the proof of Proposition \ref{propa-priori-weighted-diri}, one can use a barrier function of the form $f^{-2}$ in order to prove that $G(V)$ decays like $f^{-2}$ uniformly with respect to $\sigma$ and $V$ in a fixed ball in $X$. Similarly to the proof of Proposition \ref{a-priori-c1-bd-sol-F}, one proves the expected decay on the gradient of $G(V)$ at infinity.

\end{proof}

\subsection{Well-posedness of the homogeneous Chen-Struwe equation for small initial data}\label{section-CS-well-posed}
In this subsection, we assume that $\sigma$ is close to $1$ and hence $U_0^\sigma$ is close to a constant map in the sense that 
\[
\|U_0^{\sigma}-P\|_{L^\infty}+\sup_{x\in \R^n}(1+|x|)|\nabla U_0^{\sigma}(x)|\le C(u_0^\sigma),
\]
where $P\in \mathbb{S}^{m-1}$ and where $\lim_{\sigma\rightarrow 1}C(u_0^{\sigma})=0$. We show that for every $\varepsilon$ small enough the map $I-F^\sigma_K:B_{X}(0,\varepsilon)\to B_{X}(0,\varepsilon)$ has Leray-Schauder degree one at the origin. In other words, we show that $F^\sigma_K$ has a unique fixed point in $B_{X}(0,\varepsilon)$ if $\varepsilon$ and $1-\sigma$ are sufficiently small.

In order to see this, we note that Lemma \ref{est-Q-lemma} together with Propositions \ref{propa-priori-weighted-diri} and \ref{a-priori-c1-bd-sol-F} shows that for all $V\in B_{X}(0,\varepsilon)$, we have
\[
\|F^\sigma_K(V)\|_{X} \le C(u_0^{\sigma})+C\varepsilon^2
\]
and therefore we have indeed that $F^\sigma_K$ maps $B_{X}(0,\varepsilon)$ into itself, provided that $1-\sigma$ and $\varepsilon$ are chosen small enough.

In order to show that $F^\sigma_K$ is also a contraction on $B_{X}(0,\varepsilon)$, we invoke Lemma \ref{est-Q-lemma} together with Propositions \ref{propa-priori-weighted-diri} and \ref{a-priori-c1-bd-sol-F} again to prove that
\[
\|F^\sigma_K(V_1)-F^\sigma_K(V_2)\|_{X}\le C\varepsilon \|V_1-V_2\|_{X}
\]
and hence this shows the contraction property if we choose $\varepsilon$ small enough.
Altogether, this implies the desired result about the Leray-Schauder degree.

\section{An $\varepsilon$-regularity theorem for Chen-Struwe expanding solutions}\label{sec-varespilon-reg-thm}
We emphasize on the fact that this section does not require the target manifold $(N,g)$ to be embedded as a hypersurface of Euclidean space.

\subsection{A Bochner formula}\label{sec-boch-formula-chen-struwe}

It is a straightforward adaptation from \cite{Chen-Struwe} to get the following crucial Bochner formula:

\begin{prop}[Bochner formula]\label{Boch-Form-Chen-Str}
Let $u:\R^n\times(0,T)\rightarrow\R^m$ be a smooth solution to the Homogeneous Chen-Struwe flow:
\begin{eqnarray}
\partial_tu-\Delta u+\frac{K}{t}\chi'\left( d^2_{N}(u)\right)\nabla\left( \frac{d^2_{N}}{2}\right)(u)=0.\label{Chen-Str-equ-evo}
\end{eqnarray}
Then, the pointwise energy $e_K(u):=\frac{1}{2}\left(|\nabla u|^2+\frac{K}{t}\chi\left( d^2_{N}(u)\right)\right)$ satisfies the inequality 
\begin{eqnarray*}
(\partial_t-\Delta)e_K(u)\leq Ce_K(u)^2,
\end{eqnarray*}
on $\R^n \times (0,T)$, for some positive constant $C$ independent of $K$.
\end{prop} 
\begin{proof}
We first remark that if $d_{N}(u)\leq 2\cdot \delta_N$ then, $$\left|\nabla\left( \frac{d^2_{N}}{2}\right)(u)\right|^2=d^2_{N}(u).$$
Now,
\begin{eqnarray*}
\frac{1}{2}(\partial_t-\Delta)\chi\left( d^2_{N}(u)\right)&=&-\frac{K}{t}\chi'^2d^2_{N}(u)-\nabla\left(\chi'\nabla\left( \frac{d^2_{N}}{2}\right)(u)\right)\cdot\nabla u,\\
\frac{1}{2}(\partial_t-\Delta)|\nabla u|^2&=&-\frac{1}{t}\nabla\left(K\chi'\nabla\left( \frac{d^2_{N}}{2}\right)(u)\right)\cdot\nabla u-|\nabla^2u|^2.
\end{eqnarray*}
These computations lead to the estimate
\begin{eqnarray*}
(\partial_t-\Delta)e_K(u)+|\nabla^2u|^2+\frac{K^2}{t^2}\chi'^2d^2_{N}(u)&=&-\frac{2}{t}\nabla\left(K\chi'\nabla\left( \frac{d^2_{N}}{2}\right)(u)\right)\cdot\nabla u\\
&&-\frac{K}{t^2}\chi\left( d^2_{N}(u)\right)\\
&\leq&-\frac{2}{t}\nabla\left(K\chi'\nabla\left( \frac{d^2_{N}}{2}\right)(u)\right)\cdot\nabla u,
\end{eqnarray*}
which implies the expected result if $d_{N}(u)>2\cdot \delta_N$. If $d_{N}(u)\leq2\cdot \delta_N$, by using the fact that $\chi'$ is nonnegative we obtain
\begin{eqnarray*}
(\partial_t-\Delta)e_K(u)+|\nabla^2u|^2+\frac{K^2}{t^2}\chi'^2d^2_{N}(u)&\leq&\frac{1}{2t^2}K^2d^2_{N}(u)+c|\nabla u|^4,
\end{eqnarray*}
for some uniform positive constant $c$ independent of $K>0$.

Therefore, in all cases, this gives the expected estimate.

\end{proof}

\subsection{An energy inequality and a local entropy monotonicity formula}\label{Energy-inequ-section}
 
 We define the $L^2_{loc}$ norm at scale $R>0$ of a map $u:\R^n\rightarrow\R^m$ in $H^1_{loc}(\R^n,\R^m)$ as follows:
 \begin{eqnarray*}
\|\nabla u\|^2_{L^2_{loc,R}}:=\sup_{x_0\in \mathbb{R}^n}\fint_{B(x_0,R)}|\nabla u|^2(y)dy.
\end{eqnarray*}
It follows easily that
\begin{eqnarray}
c_n^{-1}\left(\frac{R_1}{R_2}\right)^n\|\nabla u\|^2_{L^2_{loc,R_1}}\leq\|\nabla u\|^2_{L^2_{loc,R_2}}\leq c_n \left(\frac{R_2}{R_1}\right)^n\|\nabla u\|^2_{L^2_{loc,R_1}},\quad 0<R_1\leq R_2.\label{rk-L-loc}
\end{eqnarray}
Moreover, we use the shorthand notation $\|\nabla u\|^2_{L^2_{loc}}$ for $\|\nabla u\|^2_{L^2_{loc,1}}$.

Finally, we define the rescaled energy with parameter $K>0$ for a solution $u$ to the Homogeneous Chen-Struwe flow with parameter $K>0$ by
\begin{eqnarray*}
E_{K,x_0}(u(t))&:=&\fint_{B(x_0,1)}\left(\frac{|\nabla u|^2}{2}+\frac{K}{2t}\chi\left( d^2_{N}(u)\right)\right)dx,\quad t>0,\\
E_{K,loc}(u(t))&:=&\sup_{x_0\in\R^n} E_{K,x_0}(u(t)),\quad t>0.
\end{eqnarray*}

 \begin{theo}\label{theo-l2loc-Chen-Str}
 Let $u_0:\mathbb{R}^n\rightarrow (N,g)\subset\R^m$ be in $H^1_{loc}(\R^n,\R^m)$. Let $(u(t))_{t>0}$ be a smooth solution to the Homogeneous Chen-Struwe flow with parameter $K>0$ coming out of $u_0$ such that $(E_{K,x_0}(u(t)))_{t>0}$ is continuous at $t=0$ for every $x_0\in\R^n$. Then, 
 \begin{eqnarray}
 E_{K,x_0}(u(t))&\leq& \left(1+C\left(n,m,\|\nabla u_0\|_{L^2_{loc}},t\right)\right)\|\nabla u_0\|^2_{L^2(B(x_0,1))},\quad\forall x_0\in\R^n,\label{est-a-priori-Chen-Str}\\
E_{K,loc}(u(t))&\leq& \left(1+c_n \left(e^{c_nt}-1\right)\right)\|\nabla u_0\|_{L^2_{loc}}^2, \quad t>0,\label{est-a-priori-Chen-Str-bis}
\end{eqnarray}
where $\lim_{t\rightarrow 0}C\left(n,m,\|\nabla u_0\|_{L^2_{loc}},t\right)=0$.

Moreover, the following estimate holds (it is uniform in $K$)
\begin{eqnarray*}
\sup_{x_0\in\R^n}\int_{B(x_0,1)\times(0,t)}\left(|\partial_su|^2+\frac{K}{2s^2}\chi\left( d^2_{N}(u)\right)\right)dxds\leq \left(1+c_n \left(e^{c_nt}-1\right)\right)\|\nabla u_0\|_{L^2_{loc}}^2, \quad t>0.
\end{eqnarray*}

In particular, if $u_0$ is $0-homogeneous$ and if  $u$ is an expanding solution coming out of $u_0$ smoothly, then:
 \begin{eqnarray}
E_{K,loc}(u)\leq \left(1+c_n\left(e^{c_n}-1\right)\right)\|\nabla u_0\|^2_{L^2_{loc}}.\label{rk-L-loc-Chen-Str}
\end{eqnarray}
 
 \end{theo}
 \begin{proof}
 We proceed analogously to what is done to establish an energy estimate in this setting. We multiply the Homogeneous Chen-Struwe flow equation by $\phi_{x_0}^2\partial_tu$ where $\phi_{x_0}:\mathbb{R}^n\rightarrow \mathbb{R}_+$ is a smooth function with compact support in $B(x_0,2)$ which equals $1$ on $B(x_0,1)$ and whose gradient is less than $c$, and then we integrate by parts to get
  \begin{eqnarray*}
\int_{\mathbb{R}^n}|\partial_tu|^2\phi_{x_0}^2dx&=&\int_{\R^n}<\Delta u,\partial_t u>\phi_{x_0}^2dx-\frac{K}{2t}\int_{\R^n}\partial_t\left(\chi\left( d^2_{N}(u)\right)\right)\phi_{x_0}^2dx\\
&=&-\partial_t\int_{\R^n}e_K(u)\phi_{x_0}^2dx+2<\nabla_{\nabla \phi_{x_0}} u,\phi_{x_0}\partial_tu>_{L^2}\\
&&-\frac{K}{2t^2}\int_{\R^n}\chi\left( d^2_{N}(u)\right)\phi_{x_0}^2dx\\
&\leq&-\partial_t\int_{\R^n}e_K(u)\phi_{x_0}^2dx+\frac{1}{2}\|\phi_{x_0}\partial_tu\|^2_{L^2}\\
&&+2\|\nabla_{\nabla \phi_{x_0}} u\|^2_{L^2}-\frac{K}{2t^2}\int_{\R^n}\chi\left( d^2_{N}(u)\right)\phi_{x_0}^2dx.\\
\end{eqnarray*}

 By integrating with respect to time we then obtain
\begin{eqnarray*}
\frac{1}{2}\int_{\mathbb{R}^n\times(0,t)}\left(|\partial_su|^2+\frac{K}{s^2}\chi\left( d^2_{N}(u)\right)\right)\phi_{x_0}^2dxds+E_{K,x_0}(u(t))&\leq& \frac{1}{2}\int_{\mathbb{R}^n}|\nabla u_0|^2\phi_{x_0}^2dx+\\
&&+2c^2\int_0^t\|\nabla u(s)\|^2_{L^2_{loc,2}}ds,
\end{eqnarray*}
where we used Young's inequality together with the $L^2_{loc}$ continuity of $(u(t))_{t>0}$ at $t=0$. Therefore, by remark (\ref{rk-L-loc}), one gets in particular:
\begin{eqnarray*}
\fint_{B(x_0,1)}|\nabla u(t)|^2dx\leq \|\nabla u_0\|^2_{L^2_{loc,1}}+c_n\int_0^t\|\nabla u(s)\|^2_{L^2_{loc,1}}dt,
\end{eqnarray*}
which implies:
\begin{eqnarray*}
\|\nabla u(t)\|^2_{L^2_{loc}}\leq \|\nabla u_0\|^2_{L^2_{loc}}+c_n\int_0^t\|\nabla u(s)\|^2_{L^2_{loc}}ds,\end{eqnarray*}
where $c_n$ is a positive constant that can vary from line to line depending on the dimension $n$ only.

The result now follows from Gronwall's inequality.
\end{proof}

  In the following we let $\chi_R:\mathbb{R}^n\rightarrow [0,1]$ be a cut-off function such that $\chi_R\equiv 1$ outside $B(0,R)$, $\chi_R\equiv 0$ in $B(0,R/2)$ and whose gradient satisfies $|\nabla\chi_R|=\textit{O}(R^{-1}).$ Then one has the following localised version of Theorem \ref{theo-l2loc-Chen-Str} at infinity.
 
 \begin{prop}\label{theo-l2loc-at-infty-a-priori-Chen-Str} 
  Let $u_0:\mathbb{R}^n\rightarrow N$ be in $H^1_{loc}(\R^n,\R^m)$. Let $(u(t))_{t>0}$ be a smooth solution to the Homogeneous Chen-Struwe flow coming out of $u_0$ such that $(E_{K,x_0}(u(t))_{t>0}$ is continuous at $t=0$ for every $x_0\in\R^n$. Then, if $R>0$,
 \begin{eqnarray*}
\|(\nabla u(t))\chi_R\|_{L^2_{loc}}\leq C(n,T)\left(\|(\nabla u_0)\chi_R\|_{L^2_{loc}}+\frac{\|\nabla u_0\|_{L^2_{loc}}}{R}\right), \quad 0<t\leq T.
\end{eqnarray*}
In particular, if $u_0$ is $0-homogeneous$ and if  $u$ is an expanding solution coming out of $u_0$ smoothly, then we have
 \begin{eqnarray}
\|(\nabla u)\chi_R\|_{L^2_{loc}}\leq C(n)\left(\|(\nabla u_0)\chi_R\|_{L^2_{loc}}+\frac{\|\nabla u_0\|_{L^2_{loc}}}{R}\right).\label{rk-L-loc-at-infty-Chen-Str}
\end{eqnarray}

 \end{prop}
 
 \begin{proof}
 The proof follows along the lines of the proof of Theorem \ref{theo-l2loc-Chen-Str}. Let $\phi_{x_0}$ be the cut-off function defined as previously and let us multiply the Homogeneous Chen-Struwe flow equation by $\phi_{x_0}^2\chi_R^2\partial_tu$. Then we integrate by parts in space and we integrate with respect to time. We get 

\begin{eqnarray*}
\|(\nabla u(t))\chi_R\|_{L^2_{loc}}^2&\leq& \|(\nabla u_0)\chi_R\|_{L^2_{loc}}^2+c_n\int_0^t\|(\nabla u(s))\chi_R\|_{L^2_{loc}}^2ds+\frac{c_n}{R^2}\int_0^t\|\nabla u(s)\|_{L^2_{loc}}^2ds\\
&\leq&\|(\nabla u_0)\chi_R\|_{L^2_{loc}}^2+c_n\int_0^t\|(\nabla u(s))\chi_R\|_{L^2_{loc}}^2ds+\frac{c_n}{R^2} \left(e^{c_nt}-1\right)\|\nabla u_0\|_{L^2_{loc}}^2,
\end{eqnarray*}
where we used the estimate (\ref{est-a-priori-Chen-Str-bis}) in the last line. A straightforward application of the Gronwall inequality leads to the expected result.

 \end{proof}
 In order to get a local entropy monotonicity formula, we need to localise the arguments in \cite{Chen-Struwe} since in our case the energy is infinite. For this purpose, let $z_0:=(x_0,t_0)\in\R^n\times \R_+$ and let $\phi_{x_0}:\mathbb{R}^n\rightarrow \mathbb{R}_+$ be a smooth function with compact support in $B(x_0,2)$ which equals $1$ on $B(x_0,1)$ and whose gradient is less than $c$. For $R\in(0,2^{-1}\cdot\sqrt{t_0})$, define as in [Chap. $7$, \cite{Lin-Wang}]:
 \begin{eqnarray*}
\Phi(u,z_0,R)&:=&\left.R^2\int_{\R^n}e_K(u)G_{z_0}\phi_{x_0}^2dx\right|_{t_0-R^2},\\
\Psi(u,z_0,R)&:=&\int_{t_0-4R^2}^{t_0-R^2}e_K(u)G_{z_0}\phi_{x_0}^2dxdt,
\end{eqnarray*}
where 
\begin{eqnarray*}
G_{z_0}(x,t):=\frac{1}{(4\pi|t-t_0|)^{\frac{n}{2}}}\exp\left(-\frac{|x-x_0|^2}{4|t-t_0|}\right),\quad x\in\R^n,\quad t<t_0,
\end{eqnarray*}
denotes the backward heat kernel on $\R^n$. With the notations from the previous sections: $G_{z_0}(x,t)=k_{t_0-t}(x,x_0).$
We start with a Pohozaev identity like:

\begin{prop}[Pohozaev identity]\label{Poho-Chen-Str}
Let $u:\R^n\times(0,T)\rightarrow \R^m$ be a smooth solution to the Homogeneous Chen-Struwe flow (with parameter $K>0$). Then, for any $C^1$ vector field $\zeta:\R^n\times(0,T)\rightarrow \R^m$ compactly supported in space,
\begin{eqnarray*}
<\partial_tu,\nabla_{\zeta}u>_{L^2(\R^n\times[t_1,t_2])}&=&<e_K(u),\div\zeta>_{L^2(\R^n\times[t_1,t_2])}\\
&-& \frac{1}{2}<\mathcal{L}_{\zeta}(\eucl),\nabla u\otimes\nabla u>_{L^2(\R^n\times[t_1,t_2])},\\
\end{eqnarray*}
where $\mathcal{L}_{\zeta}(\eucl)$ denotes the Lie derivative of the Euclidean metric along the vector field $\zeta$:
\begin{eqnarray*}
\frac{1}{2}<\mathcal{L}_{\zeta}(\eucl),\nabla u\otimes\nabla u>&:=&\nabla_i\zeta_j\nabla_iu_k\nabla_ju_k.
\end{eqnarray*}

 And, for any $C^1$ function $\theta:\R^n\times(0,T)\rightarrow \R$ compactly supported in space, and $0<t_1<t_2<T$,

\begin{eqnarray*}
&&\int_{L^2(\R^n\times[t_1,t_2])}\left(|\partial_tu|^2+\frac{K}{2t^2}\chi\left( d^2_{N}(u)\right)\right)\theta dxdt+\left[\int_{\R^n}e_K(u)\theta dx\right]_{t_1}^{t_2}=\\
&&\int_{\R^n\times[t_1,t_2]}e_K(u)\partial_t\theta-<\nabla_{\nabla \theta}u,\partial_tu> dxdt.
\end{eqnarray*}

\end{prop}

\begin{proof}
Define $u_{\tau}(x,t):=u(x+\tau \zeta(x,t),t+\tau\theta(x,t))$ where $\zeta(\cdot,t)$ is a smooth vector field compactly supported in space for $\tau$ small and $\theta(\cdot,t)$ is a smooth function compactly supported in space. Then,
$\partial_{\tau}u_{\tau}|_{\tau=0}=\nabla _{\zeta}u+\theta\partial_tu.$
Then, on one hand:
\begin{eqnarray*}
\int_{\R^n\times[t_1,t_2]}<\partial_tu,\nabla_{\zeta}u>dxdt&=&\int_{\R^n\times[t_1,t_2]}\left<\Delta u,\nabla_{\zeta}u\right>dxdt\\
&&-\int_{\R^n\times[t_1,t_2]}\left<\frac{K}{t}\chi'\left( d^2_{N}(u)\right)\nabla\left( \frac{d^2_{N}}{2}\right)(u),\nabla_{\zeta}u\right>dxdt.
\end{eqnarray*}
Now, by integrating by parts:
\begin{eqnarray*}
\int_{\R^n\times[t_1,t_2]}\left<\Delta u,\nabla_{\zeta}u\right>dxdt&=&\frac{1}{2}\int_{\R^n\times[t_1,t_2]}|\nabla u|^2\div\zeta-\mathcal{L}_{\zeta}(\nabla u,\nabla u)dxdt,
\end{eqnarray*}
and, similarly,
\begin{eqnarray*}
-\int_{\R^n\times[t_1,t_2]}\left<\frac{K}{t}\chi'\left( d^2_{N}(u)\right)\nabla\left( \frac{d^2_{N}}{2}\right)(u),\nabla_{\zeta}u\right>dxdt&=&\frac{1}{2}\int_{\R^n\times[t_1,t_2]} \frac{K}{t}\chi\left( d^2_{N}(u)\right)\div\zeta dxdt.
\end{eqnarray*}
Therefore, 
\begin{eqnarray*}
<\partial_tu,\nabla_{\zeta}u>_{L^2(\R^n\times[t_1,t_2])}=<e_K(u),\div\zeta>_{L^2(\R^n\times[t_1,t_2])}-\frac{1}{2}<\mathcal{L}_{\zeta}(\eucl),\nabla u\otimes\nabla u>_{L^2(\R^n\times[t_1,t_2])}.
\end{eqnarray*}

On the other hand:
\begin{eqnarray*}
\int_{\R^n\times[t_1,t_2]}<\partial_tu,\theta\partial_tu>dxdt&=&\int_{\R^n\times[t_1,t_2]}\left<\Delta u,\theta\partial_tu\right>dxdt\\
&&-\int_{\R^n\times[t_1,t_2]}\left<\frac{K}{t}\chi'\left( d^2_{N}(u)\right)\nabla\left( \frac{d^2_{N}}{2}\right)(u),\theta\partial_tu\right>dxdt\\
&=&-\int_{\R^n\times[t_1,t_2]}\partial_te_K(u)\theta dxdt-\int_{\R^n\times[t_1,t_2]}<\nabla_{\nabla \theta}u,\partial_tu>dxdt\\
&&-\int_{\R^n\times[t_1,t_2]}\frac{K}{2t^2}\chi\left( d^2_{N}(u)\right)\theta dxdt\\
&=&-\left[\int_{\R^n}e_K(u)\theta dx\right]_{t_1}^{t_2}-\int_{\R^n\times[t_1,t_2]}<\nabla_{\nabla \theta}u,\partial_tu>dxdt\\
&&+\int_{\R^n\times[t_1,t_2]}e_K(u)\partial_t\theta dxdt-\int_{\R^n\times[t_1,t_2]}\frac{K}{2t^2}\chi\left( d^2_{N}(u)\right)\theta dxdt.
\end{eqnarray*}

\end{proof}

We are now in a position to prove a local entropy monotonicity formula.
 \begin{prop}[A local entropy monotonicity formula]\label{Mono-For-Chen-Str}
 Let $u_0:\mathbb{R}^n\rightarrow (N,g)\subset \R^m$ be in $H^1_{loc}(\R^n,\R^m)$. Let $(u(t))_{t>0}$ be a smooth solution to the Homogeneous Chen-Struwe flow coming out of $u_0$ such that $(E_{K,x_0}(u(t))_{t>0}$ is continuous at $t=0$ for every $x_0\in\R^n$. Then, for any $z_0=(x_0,t_0)\in\R^n\times \R_+$ and $0<R\leq R_0<\sqrt{t_0}\leq 1$,
 \begin{eqnarray*}
&&\Phi(u,z_0,R)\leq \Phi(u,z_0,R_0)+c(R_0-R)\|\nabla u_0\|^2_{L^2(B(x_0,2))},\\
&&\Psi(u,z_0,R)\leq \Psi(u,z_0,R_0)+c(R_0-R)\|\nabla u_0\|^2_{L^2(B(x_0,2))}.
\end{eqnarray*}

 \end{prop}
 \begin{proof}
 Choose $\zeta(x,t):=G_{z_0}(x,t)\phi_{x_0}^2(x)\cdot (x-x_0)$ for $0<t<t_0$ and $x\in\R^n$. Then,
 \begin{eqnarray*}
\div \zeta&=&\phi^2_{x_0}\div(G_{z_0}\cdot (x-x_0))+<\nabla \phi_{x_0}^2,G_{z_0}\cdot (x-x_0)>\\
&=&\phi^2_{x_0}\left(-\frac{|x-x_0|^2}{2|t-t_0|}+n\right)G_{z_0}+<\nabla \phi_{x_0}^2,G_{z_0}\cdot (x-x_0)>,\\
\nabla_i(G_{z_0}(x-x_0))_j&=&\left(-\frac{(x-x_0)_i(x-x_0)_j}{2|t-t_0|}+\delta_{ij}\right)G_{z_0},\\
\frac{1}{2}<\mathcal{L}_{\zeta}(\eucl),\nabla u\otimes\nabla u>&:=&\left(-\frac{|\nabla_{x-x_0}u|^2}{2|t-t_0|}+|\nabla u|^2\right)G_{z_0}\phi^2_{x_0}+<\nabla_{\nabla \phi^2_{x_0}}u,\nabla_{x-x_0}u>G_{z_0}.
\end{eqnarray*}
Therefore, by applying Proposition \ref{Poho-Chen-Str}, one gets
\begin{eqnarray*}
&&\frac{1}{2}\int_{\R^n\times[t_1,t_2]}\left[<\partial_tu,\nabla_{x-x_0}u>-\left(-\frac{|x-x_0|^2}{2|t-t_0|}+n\right)e_K(u)-\frac{|\nabla_{x-x_0}u|^2}{2|t-t_0|}+|\nabla u|^2\right]G_{z_0}\phi^2_{x_0}dxdt=\\
&&\frac{1}{2}\int_{\R^n\times[t_1,t_2]}\left(<\nabla \phi_{x_0}^2,x-x_0>e_K(u)-<\nabla_{\nabla \phi^2_{x_0}}u,\nabla_{x-x_0}u>\right)G_{z_0}dxdt.
\end{eqnarray*}
Now, by using Proposition \ref{Poho-Chen-Str} with $\theta(x,t):=(t_0-t)G_{z_0}\phi_{x_0}^2$ for $0<t<t_0$ and $x\in\R^n$, one obtains
\begin{eqnarray*}
&&\partial_t\theta=\left(\frac{n-2}{2}-\frac{|x-x_0|^2}{4|t-t_0|}\right)G_{z_0}\phi^2_{x_0},\\
&&\nabla \theta=-\frac{x-x_0}{2}G_{z_0}\phi_{x_0}^2+(t_0-t)G_{z_0}\nabla\phi_{x_0}^2,\\
&&\int_{\R^n\times[t_1,t_2]}|\partial_tu|^2(t_0-t)G_{z_0}\phi_{x_0}^2dxdt+\left[(t_0-t)\int_{\R^n}e_K(u)G_{z_0}\phi_{x_0}^2dx\right]_{t_1}^{t_2}\leq\\
&&\int_{\R^n\times[t_1,t_2]}\left\{e_K(u)\left(\frac{n-2}{2}-\frac{|x-x_0|^2}{4|t-t_0|}\right)+\frac{1}{2}<\partial_tu,\nabla_{x-x_0}u>\right\}G_{z_0}\phi^2_{x_0}dxdt\\
&&-\int_{\R^n\times[t_1,t_2]}(t_0-t)<\partial_tu,\nabla_{\nabla \phi_{x_0}^2}u>G_{z_0}dxdt.
\end{eqnarray*}
Subtracting the two previous identities yields
\begin{eqnarray*}
&&\int_{\R^n\times[t_1,t_2]}(t_0-t)\left|\partial_tu-\nabla_{\frac{x-x_0}{2(t_0-t)}}u\right|^2G_{z_0}\phi^2_{x_0}dxdt+\left[(t_0-t)\int_{\R^n}e_K(u)G_{z_0}\phi_{x_0}^2dx\right]_{t_1}^{t_2}\leq\\
&&-\int_{\R^n\times[t_1,t_2]}\frac{K}{2t}\chi\left( d^2_{N}(u)\right)G_{z_0}\phi_{x_0}^2dxdt-\int_{\R^n\times[t_1,t_2]}(t_0-t)<\partial_tu,\nabla_{\nabla \phi_{x_0}^2}u>G_{z_0}dxdt\\
&&-\frac{1}{2}\int_{\R^n\times[t_1,t_2]}\left(<\nabla \phi_{x_0}^2,x-x_0>e_K(u)-<\nabla_{\nabla \phi^2_{x_0}}u,\nabla_{x-x_0}u>\right)G_{z_0}dxdt\\
&\leq&-\frac{1}{2}\int_{\R^n\times[t_1,t_2]}\left((t_0-t)<\nabla_{\nabla \phi^2_{x_0}}u,\partial_tu-\nabla_{\frac{x-x_0}{2(t_0-t)}}u>+<\nabla \phi_{x_0}^2,x-x_0>e_K(u)\right)G_{z_0}dxdt\\
&\leq&\frac{1}{2}\int_{\R^n\times[t_1,t_2]}(t_0-t)\left|\partial_tu-\nabla_{\frac{x-x_0}{2(t_0-t)}}u\right|^2G_{z_0}\phi^2_{x_0}dxdt+c\int_{\supp(\nabla \phi_{x_0})\times[t_1,t_2]}e_K(u)dxdt.
\end{eqnarray*}
On $\supp(\nabla \phi_{x_0})$ we have $G_{z_0}(x,t)\leq C$ for all $t\in[0,t_0)$ and therefore the last term can be bounded as follows with the help of Theorem \ref{theo-l2loc-Chen-Str} as follows
\begin{eqnarray*}
\int_{\supp(\nabla \phi_{x_0})\times[t_1,t_2]}e_K(u)dxdt\leq c\|\nabla u_0\|^2_{L^2(B(x_0,2))}.
\end{eqnarray*}
This in turn implies the expected monotonicity result for $\Phi$. The monotonicity result for $\Psi$ follows from the one for $\Phi$.

 \end{proof}

\subsection{An $\varepsilon$-regularity theorem}\label{sec-eps-reg-theo}

\begin{theo}\label{eps-reg-Chen-Str}
Let $u_0:\R^n\rightarrow (N,g)\subset\R^m$ be a $0$-homogeneous Lipschitz map. 
Then there exists a radius $R=R(\|\nabla u_0\|_{L^2_{loc}},n,m)>0$ and a constant $C=C(\|\nabla u_0\|_{L^2_{loc}},n,m)>0$ such that if $u$ is a smooth solution of the Homogeneous Chen-Struwe flow with parameter $K>0$ coming out of $u_0$ and such that $(E_{K,x_0}(u(t))_{t>0}$ is continuous at $t=0$ for any $x_0\in\R^n$, $u$ satisfies
\begin{eqnarray*}
e_K(u)(x,1)\leq\frac{C}{|x|^2},\quad |x|\geq R.
\end{eqnarray*}

Moreover, there exists a constant $\varepsilon_0>0$ depending on $n$ and $N$ only such that if for some $ R\in(0,\min\{\sqrt{t_0},1\})$, $z_0=(x_0,t_0)\in\R^n\times\R_+$, $u$ satisfies
$$\Psi(u,z_0,R)<\varepsilon_0,$$
then $$\sup_{P_{\delta R}(z_0)}e_K(u)\leq C(\delta R)^{-2},$$
for some universal positive constant $C$ and some positive constant $\delta$ depending on $n$, $m$, $\|\nabla u_0\|_{L^2_{loc}}$ and $\min\{R,1\}$. 
\end{theo}

\begin{proof}
The proof is a straightforward adaptation of the corresponding proof in the case of the Chen-Struwe flow. We mention the main steps: see [Lemma $2.4$, \cite{Chen-Struwe}] for more details.

 First of all, let $(x_1,t_1)=:z_1\in\R^n\times\R_+$, $0<R<2^{-1}\sqrt{t_1}$ and let $r_1:=2\delta R$ with $\delta\in(0,1/4)$ where $\delta$ will be defined later. Let $r,\sigma\in[0,r_1)$ such that $r+\sigma<r_1$ and let $z_0:=(x_0,t_0)\in P_r(z_1).$ Thanks to the monotonicity formula from Proposition \ref{Mono-For-Chen-Str}, then one shows that for a given positive $\varepsilon$, there exists a positive $\delta(\varepsilon)$ such that
 \begin{eqnarray}
\sigma^{-n}\int_{P_{\sigma}(z_0)}e_K(u)dxdt\leq c\Psi(u,z_1,R)+c( (R-\sigma)+\varepsilon)\|\nabla u_0\|^2_{L^2(B(x_1,2))}.\label{all-scale-energy}
\end{eqnarray}

Now, by smoothness of the solution, there exists $\sigma_0\in[0,r_1)$ and $(x_0,t_0)\in \overline{P_{\sigma_0}(z_1)}$ such that
\begin{eqnarray*}
(r_1-\sigma_0)^2e_K(u)(x_0,t_0)=\max_{0\leq\sigma\leq r_1}(r_1-\sigma)^2\sup_{\overline{P_{\sigma}(z_1)}} e_K(u).
\end{eqnarray*}
If one defines $\rho_0:=\frac{1}{2}(r_1-\sigma_0)$, $r_0:=\sqrt{e_0}\rho_0$, and $$v(x,t):=u\left(\frac{x}{\sqrt{e_0}}+x_0,\frac{t}{e_0}+t_0\right),\quad (x,t)\in P_{r_0}(0,0),$$
then, $v$ satisfies
\begin{eqnarray*}
&&(\partial_t-\Delta)v=-\frac{K}{e_0t_0+t}\chi'\left( d^2_{N}(v)\right)\nabla\left( \frac{d^2_{N}}{2}\right)(v)=0,\\
&&e_K(v)(0,0)=1,\quad \sup_{P_{r_0}(0,0)}e_K(v)\leq 4.
\end{eqnarray*}
By Proposition \ref{Boch-Form-Chen-Str} we obtain
\begin{eqnarray*}
(\partial_t-\Delta)e_K(v)\leq 4Ce_K(v),\quad \mbox{on $P_{r_0}(0,0)$},
\end{eqnarray*}
and Moser's Harnack inequality together with (\ref{all-scale-energy}) shows that $r_0\leq 1$ if $\Psi(u,z_1,R)$ is small enough (independently of $K$).

The final step consists in applying Moser's Harnack inequality again to $v$ in order to get
\begin{eqnarray*}
\max_{0\leq\sigma\leq r_1}(r_1-\sigma)^2\sup_{P_{\sigma}}e_K(u)&\leq& 4\rho_0^2e_0=4r_0^2\\
&\leq& c\rho_0^{-n}\int_{P_{\rho_0}(x_0,t_0)}e_K(u)dxdt\\
&\leq& c\Psi(u,z_1,R)+c(R+\varepsilon)\|\nabla u_0\|^2_{L^2(B(x_1,2))}.
\end{eqnarray*}
 
 The result follows by noticing that if $u_0$ is Lipschitz then by the chain rule
 \begin{eqnarray*}
\|\nabla u_0\|_{L^2(B(x_1,2))}\leq \frac{C}{1+|x_1|}.
\end{eqnarray*}
In particular, thanks to Proposition \ref{theo-l2loc-at-infty-a-priori-Chen-Str}, if $x_1$ is sufficiently far from the origin, by choosing $t_1:=1$, $R:=1/4$, both $\Psi(u,z_1,1/4)$ and $\|\nabla u_0\|_{L^2(B(x_1,2))}$ can be made arbitrarily small, independently of $K$.

The second statement can be proved similarly.

\end{proof}
\section{A priori estimates for Chen-Struwe expanding solutions} \label{sec-a-priori-chen-struwe-expanders}
\subsection{$C^0$ a priori estimates}\label{sec-a-priori-chen-struwe}

We start by establishing an a priori $C^0$ bound for Chen-Struwe expanding solutions $U_K^{\sigma}$  uniform in $K>0$ and $\sigma\in[0,1]$.

\begin{prop}\label{a-priori-C0-Chen-Str}
 There is a positive constant $M$ uniform in $\sigma\in[0,1]$ and $K>0$ such that if $V\in X$ satisfies $F_K^{\sigma}(V)=V$, then $\|V\|_{C^0}\leq M.$
 \end{prop}
 \begin{proof}
As $F_K^{\sigma}(V)=V$ we know that $U^{\sigma}:=U_0^{\sigma}+V$ solves the static Homogeneous Chen-Struwe flow $$\Delta_fU^{\sigma}=K\chi'\left( d^2_{N}(U^{\sigma})\right)\nabla\left( \frac{d^2_{N}}{2}\right)(U^{\sigma}).$$ In particular,

\begin{eqnarray*}
\Delta_f|U^{\sigma}|^2\geq 2|\nabla U^{\sigma}|^2+2\left<K\chi'\left( d^2_{N}(U^{\sigma})\right)\nabla\left( \frac{d^2_{N}}{2}\right)(U^{\sigma}),U^{\sigma}\right>.
\end{eqnarray*}

Next we fix a radius $R>0$ and consider $\max_{B(0,R)}|U^{\sigma}|$. If this maximum is attained at an interior point $x_R$ of $B(0,R)$, then we consider two cases.

 Either $d_{N}(U^{\sigma}(x_R))\leq 2\cdot \delta_N$ which implies that $\max_{B(0,R)}|U^{\sigma}|$ is uniformly bounded by the triangular inequality.
 
 Or $d_{N}(U^{\sigma}(x_R))> 2\cdot \delta_N$ and this implies by the strong maximum principle applied to the previous differential inequality that $|U^{\sigma}|$ is constant and that $\nabla U^{\sigma}=0$ on a neighborhood of $x_R$. Therefore, $U^{\sigma}$ is constant on a neighborhood of $x_R$. By connectedness, $U^{\sigma}$ is constant on $B(0,R)$. As $U^{\sigma}$ converges to $u_0^{\sigma}$ at infinity, the term $\sup_{\partial B(0,R)}\bar{d}_N(U^{\sigma})$ goes to $0$ as $R$ goes to $+\infty$. Consequently, this case is impossible if $R$ is large enough.
 
 This discussion ends the proof of the $C^0$ estimate. Moreover, this last fact also yields the desired estimate if the maximum is attained on the boundary.

\end{proof}
By interior parabolic Schauder estimates, one has the following corollary.
  \begin{coro}\label{para-reg-coro-Che-Str}
 For any $k\geq 0$, there is a positive constant $M(K,k)$ uniform in $\sigma\in[0,1]$ and $K>0$ such that if $V\in X$ is a fixed point of the map $F_K^{\sigma}$ then $\|V\|_{C^k}\leq M(K,k).$
 \end{coro}
 \begin{rk}
 The constants $M(K,k)$ in Corollary \ref{para-reg-coro-Che-Str} may depend on $K$.
 \end{rk}
 
\subsection{A priori $C^0$ estimate at infinity}\label{sec-a-priori-weighted-est}
The purpose of this section is to establish a priori $C^0$ weighted estimates for Chen-Struwe expanding solutions that are  uniform in $\sigma\in[0,1]$. The bounds might depend on the parameter $K$.

\begin{prop}\label{a-priori-C0-weighted-Che-Str}
 There is a positive constant $M$ uniform in $\sigma\in[0,1]$ such that if $V\in X$ is a fixed point of $F_K^{\sigma}$ then $\|fV\|_{C^0}\leq M.$
 \end{prop}
 \begin{rk}
 Because of Proposition \ref{a-priori-C0-Chen-Str}, it suffices to show this a priori bound outside a ball of radius independent of $\sigma\in[0,1]$.
 \end{rk}

\begin{proof}
Since $V$ is fixed point of the map $F_K^{\sigma}$ it follows that
\begin{eqnarray}
\Delta_fV&=&K\chi'(d_N^2(U_0^{\sigma}+V))\nabla \left(\frac{d_N^2}{2}\right)(U_0^{\sigma}+V)\label{eq-fixed-pt}\\
&=&K\bar{d}_N(U_0^{\sigma}+V)\chi'(d_N^2(U_0^{\sigma}+V))\nabla \bar{d}_N(U_0^{\sigma}+V)\\
&=&\textit{O}(f^{-1/2}),
\end{eqnarray}

where $\textit{O}(\cdot)$ is uniform in $\sigma\in[0,1]$ and where we used Theorem \ref{eps-reg-Chen-Str} in order to ensure that $\bar{d}_N(U_0^{\sigma}+V)=\textit{O}(f^{-1/2}).$ Therefore, by using $f^{-1/2}$ as a barrier, one gets a first a priori bound on the decay of $V$, namely there exists a positive constant $M$ independent of $\sigma\in[0,1]$ such that 
\begin{eqnarray}
\|f^{1/2}V\|_{C^0}\leq M.\label{first-a-priori-rough-weighted-est}
\end{eqnarray}
Now, we use the special structure of the nonlinearities of equation (\ref{eq-fixed-pt}) together with the previous a priori estimate (\ref{first-a-priori-rough-weighted-est}) and Lemma \ref{lemma-first-app-Chen-Str}.
Outside a ball of radius sufficiently large (but independent of $V$ and $\sigma\in[0,1]$), one has
\begin{eqnarray*}
\Delta_f|V|^2&=&2|\nabla V|^2+2K\bar{d}_N(U_0^{\sigma}+V)\left<\nabla \bar{d}_N(U_0^{\sigma}+V),V\right>,\\
|d_{U_0^{\sigma}+V}\bar{d}_N(V)-d_{U_0^{\sigma}}\bar{d}_N(V)|&\leq& C(N,\|U_0^{\sigma}\|_{L^{\infty}})|V|^2,\\
|\bar{d}_N(U_0^{\sigma}+V)-\bar{d}_N(U_0^{\sigma})-d_{U_0^{\sigma}}\bar{d}_N(V)|&\leq&C(N,\|U_0^{\sigma}\|_{L^{\infty}})|V|^2,
\end{eqnarray*}
which implies:
\begin{eqnarray*}
\Delta_f|V|^2&\geq&2|\nabla V|^2-C(N,\|U_0^{\sigma}\|_{L^{\infty}})K\bar{d}_N(U_0^{\sigma})|V|+2K\left(d_{U_0^{\sigma}}\bar{d}_N(V)\right)^2-C(N,\|U_0^{\sigma}\|_{L^{\infty}})K|V|^3\\
&\geq&2|\nabla V|^2-C(N,\|U_0^{\sigma}\|_{L^{\infty}})K\bar{d}_N(U_0^{\sigma})|V|-C(N,\|U_0^{\sigma}\|_{L^{\infty}})K|V|^3\\
&\geq&2|\nabla V|^2-\textit{O}(f^{-1})|V|,
\end{eqnarray*}

In particular, by the Kato inequality, 
\begin{eqnarray*}
\Delta_f|V|\geq -\textit{O}(f^{-1}),
\end{eqnarray*}
when $|V|$ does not vanish. 

In general, one can use the regularization of $|V|$ of the form $V_{\varepsilon}:=\sqrt{|V|^2+\varepsilon^2}$ where $\varepsilon$ is positive which satisfies the same differential inequality
\begin{eqnarray*}
\Delta_fV_{\varepsilon}\geq -\textit{O}(f^{-1}).
\end{eqnarray*}
Now, as $\Delta_ff^{-1}=-(1+\textit{o}(1))f^{-1}$, one can use $f^{-1}$ as a barrier function as follows
\begin{eqnarray*}
\Delta_f(V_{\varepsilon}-Af^{-1})>0,
\end{eqnarray*}
outside a sufficiently large ball $B(0,R)$ independent of $\sigma\in[0,1]$ and for any sufficiently large constant $A$.
In particular, as $V_{\varepsilon}$ is bounded independently of $\varepsilon\in(0,1]$ and of $\sigma\in[0,1]$ (and of $K$), one can choose a constant $A$ sufficiently large such that on the boundary $\partial B(0,R)$, $\sup_{\partial B(0,R)}V_{\varepsilon}-Af^{-1}<0$. By applying the maximum principle to the function $V_{\varepsilon}-Af^{-1}$, one gets
\begin{eqnarray*}
\sup_{\R^n\setminus B(0,R)}\left\{V_{\varepsilon}-Af^{-1}\right\}\leq \limsup_{+\infty}\left\{V_{\varepsilon}-Af^{-1}\right\}=\varepsilon.
\end{eqnarray*}
Since $A$ and $R$ can be chosen independently of $\varepsilon\in(0,1]$, one can pass to the limit in the previous inequality as $\varepsilon$ goes to $0$ to get the expected result.

\end{proof}

\subsection{Weighted $C^1$ estimate}\label{sec-a-priori-C1-weighted-Che-Str}
In this section, we prove a priori $C^1$ weighted estimates for Chen-Struwe expanding solutions that are uniform in the parameter $\sigma\in[0,1]$. As in Section \ref{sec-a-priori-weighted-est}, the bounds might depend on the parameter $K$.
\begin{prop}\label{a-priori-C1-weighted-Che-Str}
 There is a positive constant $M$ uniform in $\sigma\in[0,1]$ such that if $V\in X$ is a fixed point of $F_K^{\sigma}$ then $\|f^{3/2}\nabla V\|_{C^0}\leq M.$
 \end{prop}
 \begin{rk}
 Because of Proposition \ref{a-priori-C0-Chen-Str}, it suffices to show this a priori bound outside a ball of radius independent of $\sigma\in[0,1]$.
 \end{rk}

\begin{proof}
 We first establish the evolution equation satisfied by the gradient $\nabla V$. Since $U:=U_0^{\sigma}+V$ is an expanding solution of the homogeneous Chen-Struwe flow and because of the previous remark, the gradient $\nabla V$ satisfies the following equation outside a sufficiently large ball independent of $\sigma$
 
 \begin{eqnarray*}
\Delta_f\nabla V&=&-\frac{\nabla V}{2}+K\nabla\left(\nabla\left( \frac{d^2_{N}}{2}\right)(U)\right).
\end{eqnarray*}
More precisely, in coordinates, this gives:

\begin{eqnarray*}
\Delta_f\nabla_i V_j&=&-\frac{\nabla_i V_j}{2}+K\nabla_i\left(\nabla_j\left( \frac{d^2_{N}}{2}\right)(U)\right)\\
&=&-\frac{\nabla_i V_j}{2}+K\nabla_i\left( \bar{d}_N(U)\nabla_j \bar{d}_N(U)\right)\\
&=&-\frac{\nabla_i V_j}{2}+K\left(\nabla_i (\bar{d}_N(U))\cdot(\nabla_j \bar{d}_N)(U)+\bar{d}_N(U)\cdot\nabla_i(\nabla_j \bar{d}_N(U))\right).
\end{eqnarray*}

By Taylor expansion of order $2$ together with Proposition \ref{a-priori-C0-weighted-Che-Str}, Theorem \ref{eps-reg-Chen-Str} and Lemma \ref{lemma-first-app-Chen-Str} we have
 \begin{eqnarray*}
|\nabla(\bar{d}_N(U)-\bar{d}_N(U_0^{\sigma})-<\nabla \bar{d}_N(U_0^{\sigma}),V>)|&\leq& C(N)(|\nabla V||V|+|\nabla U_0^{\sigma}||V|^2),\\
&\leq& \textit{O}(f^{-1})|\nabla V|+\textit{O}(f^{-3/2}),
\end{eqnarray*}
and,

\begin{eqnarray*}
(\nabla (\nabla \bar{d}_N(U_0^{\sigma})))(V)&=&\textit{O}(f^{-3/2}),\\
\bar{d}_N(U)\cdot\nabla_i(\nabla_j \bar{d}_N(U))&=&\textit{O}(f^{-3/2}),\\
\nabla \bar{d}_N(U)&=&\nabla \bar{d}_N(U_0^{\sigma})+\textit{O}(V)=\nabla \bar{d}_N(U_0^{\sigma})+\textit{O}(f^{-1}).
\end{eqnarray*}

Now, we can end the argument by discarding the nonnegative terms that are quadratic in the gradient $\nabla V$ as follows
\begin{eqnarray*}
\Delta_f|\nabla V|^2&\geq&2|\nabla^2 V|^2-(1+\textit{O}(f^{-1}))|\nabla V|^2-\textit{O}(f^{-3/2})|\nabla V|\\
&&+2K\nabla (\bar{d}_N(U_0^{\sigma}))<\nabla \bar{d}_N(U), \nabla V>+2K<\nabla \bar{d}_N(U_0^{\sigma})(\nabla V),\nabla \bar{d}_N(U)(\nabla V)>\\
&\geq&2|\nabla^2 V|^2-(1+\textit{O}(f^{-1}))|\nabla V|^2-\textit{O}(f^{-3/2})|\nabla V|\\
&&+2K<\nabla \bar{d}_N(U_0^{\sigma})(\nabla V),\nabla \bar{d}_N(U_0^{\sigma})(\nabla V)>\\
&\geq&2|\nabla^2 V|^2-(1+\textit{O}(f^{-1}))|\nabla V|^2-\textit{O}(f^{-3/2})|\nabla V|,\\
&\geq&2|\nabla^2 V|^2-|\nabla V|^2-\textit{O}(f^{-3/2})|\nabla V|,\\
\end{eqnarray*}
where we used in the last line the fact that $\nabla V=\textit{O}(f^{-1/2})$ a priori, thanks to Theorem \ref{eps-reg-Chen-Str}.

By considering the function $f^{1/2}|\nabla V|_{\varepsilon}-Af^{-1}$ where $|\nabla V|_{\varepsilon}$ denotes a regularization of the norm of the gradient $|\nabla V|$ and where $A$ a positive constant large enough depending eventually on $K$ but independent of $\sigma\in[0,1]$ and $\varepsilon\in(0,1]$, one can prove the expected a priori estimate on $\nabla V$ with the help of the maximum principle.

\end{proof}

\subsection{$L^2_{loc}$ convergence at $t=0$.}\label{section-$L^2_{loc}$ convergence-Che-Str}
 In this section, we investigate the $L^2_{loc}$ continuity at $t=0$ of the expanding solution of the homogeneous Chen-Struwe equation we produced in the previous sections. Since $U(t)=k_t\ast u_0+V(t)$ with $\nabla^iV(x,1)=\textit{O}((1+|x|)^{-2-i})$ with $i=0,1$, i.e. $V\in X$. It suffices to prove the $L^2$ convergence on a ball $B(0,R)$ centered at the origin with radius $R$. We claim that: $$\lim_{t\rightarrow 0}\|\nabla V\|_{L^2(B(0,R))}(t)=0.$$
 Indeed, since $V\in X$, $\nabla V$ decays at least quadratically, i.e. $$\nabla V(x,t)=\textit{O}\left(\frac{1}{\sqrt{t}}\frac{1}{\left(\frac{|x|}{\sqrt{t}}+1\right)^{2}}\right).$$ Therefore,
 \begin{eqnarray*}
\|\nabla V\|^2_{L^2(B(0,R))}(t)\leq C(n,u_0)\int_{B(0,R)}\frac{t}{(|x|+\sqrt{t})^4}dx=C(n,u_0) \int_0^R\frac{t}{(r+\sqrt{t})^4}r^{n-1}dr,
\end{eqnarray*}
for some positive constant $C(n,u_0)$.

Now, if $n\geq5$, 
\begin{eqnarray*}
\|\nabla V\|^2_{L^2(B(0,R))}(t)\leq C(n,u_0,R)t\rightarrow 0,\quad \mbox{as $t\rightarrow 0^+$.}
\end{eqnarray*}
 If $n=4$, \begin{eqnarray*}
\|\nabla V\|^2_{L^2(B(0,R))}(t)\leq C(n,u_0)\ln\left(\frac{R}{\sqrt{t}}+1\right)t\rightarrow 0,\quad \mbox{as $t\rightarrow 0^+$.}
\end{eqnarray*}
If $n=3$,
\begin{eqnarray*}
\|\nabla V\|^2_{L^2(B(0,R))}(t)\leq C(n,u_0)\sqrt{t}\rightarrow 0,\quad \mbox{as $t\rightarrow 0^+$.}
\end{eqnarray*}

\subsection{Proof of Theorem \ref{Chen-Str-app}}
In this section, we give the proof of Theorem \ref{Chen-Str-app}.
\begin{proof}[Proof of Theorem \ref{Chen-Str-app}]
Let $K>0$ and let $u_0:\R^n\rightarrow N^{m-1}\subset\R^m$ be a $0$-homogeneous map $u_0$ as in the statement of Theorem \ref{Chen-Str-app}.

Thanks to Propositions \ref{propa-priori-weighted-diri} and \ref{a-priori-c1-bd-sol-F}, the map $F_K:X\times[0,1]\rightarrow X$ is a well-defined compact continuous map and Proposition \ref{map-cont-cpct} ensures that $F_K$ is a compact and continuous map.

Moreover, the Leray-Schauder degree of $I-F_K^{\sigma}:B_{X}(0,\varepsilon)\rightarrow B_{X}(0,\varepsilon)$ is $1$ when $\sigma$ is close to $1$, for some positive $\varepsilon$ by Lemma \ref{est-Q-lemma} combined with Section \ref{section-CS-well-posed}.

Finally, there is a positive constant $M$ (uniform in $\sigma\in[0,1]$) such that if $V\in X$ is such that $F_K^{\sigma}(V)=V$ then $\|V\|_{X}\leq M$ by the combination of Propositions \ref{a-priori-C0-Chen-Str}, \ref{a-priori-C0-weighted-Che-Str} and \ref{a-priori-C1-weighted-Che-Str} proved in Section \ref{sec-a-priori-chen-struwe-expanders} with the help of Section \ref{sec-varespilon-reg-thm}.

As a consequence of the Leray-Schauder fixed point Theorem, for each positive $K$, the map $F_K^0:X\rightarrow X$ has a fixed point $V_K\in X$, i.e. the map $U_K:=U_0+V_K$ is a smooth solution to (\ref{Chen-Str-equ}) by Section \ref{Section-Fixed-point-formulation}.

Finally, Section \ref{section-$L^2_{loc}$ convergence-Che-Str} ensures that the time-dependent expanding solution $u_K(t)$ converges strongly to $u_0$ as $t$ goes to $0$ in $H^1_{loc}(\R^n)$.
\end{proof}
\section{Proof of Theorem \ref{main-theo}}\label{sec-lip-end-proof}
In this section, we give the proof of Theorem \ref{main-theo}:

\begin{proof}[Proof of Theorem \ref{main-theo}]
Let $(u_0^{\varepsilon})_{\varepsilon\in(0,1)}$  be a sequence of $0$-homogeneous maps $u_0^{\varepsilon}:\mathbb{R}^n\rightarrow N\subset\R^m$ in $C^3_{loc}(\R^n\setminus\{0\})$ converging to $u_0$ in the $C^0$ topology as $\varepsilon$ goes to $0$,  such that 
\begin{eqnarray}
\limsup_{\varepsilon\rightarrow 0}\Lip(u_0^{\varepsilon})\leq \Lip(u_0).\label{app-lip-maps}
\end{eqnarray}

Let $K>0$ and let $(U_K^{\varepsilon})_{\varepsilon\in(0,1)}$ be a sequence of smooth Chen-Struwe expanding solutions with fixed parameter $K$ coming out of $u_0^{\varepsilon}$ given by Theorem \ref{Chen-Str-app}.
By Theorem \ref{Chen-Str-app}, we know there exists a radius $R=R(\|\nabla u_0^{\varepsilon}\|_{L^2_{loc}},n,m)>0$ and a constant $C=C(\|\nabla u_0^{\varepsilon}\|_{L^2_{loc}},n,m)>0$ such that,
\begin{eqnarray}
&&|e_K(U_K^{\varepsilon})|(x)\leq \frac{C}{|x|^2},\quad |x|\geq R,\label{equ:1}\\
&&\|e_K(u_K^{\varepsilon})(t)\|_{L^1(B(x_0,1))}\leq C\left(n,m,\|\nabla u_0^{\varepsilon}\|_{L^2_{loc}(\R^n)},t\right)\|\nabla u_0^{\varepsilon}\|^2_{L^2(B(x_0,1))},\quad\forall x_0\in\R^n,\label{equ:2}\\
&&\|\partial_tu_K^{\varepsilon}\|_{L^2((0,t),L^2_{loc}(\R^n))}\leq C(n,m,t)\|\nabla u_0^{\varepsilon}\|_{L^2_{loc}(\R^n)},\label{equ:3}
\end{eqnarray}
where $\lim_{t\rightarrow 0}C\left(n,m,\|\nabla u_0^{\varepsilon}\|_{L^2_{loc}(\R^n)},t\right)=\lim_{t\rightarrow}C(n,m,t)=1$.
According to (\ref{app-lip-maps}), there exist constants $C$ and $R$ uniform in $\varepsilon$ such that the previous inequalities hold.

Fix $\varepsilon\in(0,1)$. As in \cite{Chen-Struwe}, there exists a subsequence (still denoted by $(u_{K}^{\varepsilon})_{K>0}$) converging weakly to a map $u^{\varepsilon}:\R^n\times\R_+\rightarrow \R^m$ as $K$ tends to $+\infty$ such that

\begin{eqnarray*}
&&u^{\varepsilon}(\lambda x,\lambda^2t)=u^{\varepsilon}(x,t), \quad \forall \lambda >0,\quad \text{a.e. $(x,t)\in \R^n\times \R_+$},\\
&&\nabla u_K^{\varepsilon}\rightharpoonup\nabla u^{\varepsilon},\quad\text{weakly$^*$ in $L^{\infty}(\R_+,L^2_{loc}(\R^n))$},\\
&&\partial_tu_K^{\varepsilon}\rightharpoonup \partial_tu^{\varepsilon},\quad\text{weakly in $L^{2}_{loc}(\R_+,L^2_{loc}(\R^n))$},\\
&&u_K^{\varepsilon}\rightarrow u^{\varepsilon},\quad\text{ in $L^2_{loc}(\R^n)$},\\
&&u_K^{\varepsilon}\rightarrow u^{\varepsilon},\quad\text{ in $C^{0,\beta}_{loc}(\R^n\setminus B(0,R))$, for all $\beta\in(0,1)$.}
\end{eqnarray*}
The last point is due to Arzela-Ascoli Theorem together with the estimate (\ref{equ:1}). Moreover, thanks to (\ref{equ:2}), $u^{\varepsilon}\in N$ a.e.. 

This implies in particular that $U^{\varepsilon}$ is a Lipschitz function outside $B(0,R)$. To sum it up, we have obtained that 
\begin{eqnarray}
&&|\Lip(U^{\varepsilon})|(x)\leq \frac{C}{|x|},\quad |x|\geq R,\label{equ:1:bis}\\
&&\|\nabla u^{\varepsilon}(t)\|_{L^2(B(x_0,1))}\leq C\left(n,m,\|\nabla u_0^{\varepsilon}\|_{L^2_{loc}(\R^n)},t\right)\|\nabla u_0^{\varepsilon}\|_{L^2(B(x_0,1))},\quad\forall x_0\in\R^n,\label{equ:2:bis}\\
&&\|\partial_tu^{\varepsilon}\|_{L^2((0,t),L^2_{loc}(\R^n))}\leq C(n,m,t)\|\nabla u_0^{\varepsilon}\|_{L^2_{loc}(\R^n)},\label{equ:3:bis}
\end{eqnarray}
In particular, if one shows that $u^{\varepsilon}(\cdot,t):=U^{\varepsilon}(\cdot/\sqrt{t})$ converges weakly to $u_0^{\varepsilon}$ then, $$\liminf_{t\rightarrow 0}\|\nabla u^{\varepsilon}(t)\|_{L^2(B(x_0,1))}\geq \|\nabla u_0^{\varepsilon}\|_{L^2(B(x_0,1))}.$$ By combining this fact with (\ref{equ:3}), one ends up by proving that $u^{\varepsilon}(t)$ converges to $u_0^{\varepsilon}$ in $H^1_{loc}(\R^n)$ (in the strong sense).

\begin{claim}
$u^{\varepsilon}(t)\rightharpoonup u_0$ as $t\rightarrow 0$.\\
\end{claim}

To prove this statement, let $\psi_{x_0}:\R^n\rightarrow\R^m$ be a smooth map with compact support in $B(x_0,1)$ and let $0<s<t$. Then, for $K>0$:

\begin{eqnarray*}
\left|\int_{\R^n}<u^{\varepsilon}_K(t)-u^{\varepsilon}_K(s),\psi_{x_0}>dx\right|&=&\left|\int_s^t\int_{\R^n}<\partial_{\tau}u^{\varepsilon}_K,\psi_{x_0}>dxd\tau\right|\\
&\leq&\int_s^t\int_{B(x_0,1)}|\partial_{\tau}u^{\varepsilon}_K||\psi_{x_0}|dxd\tau\\
&\leq&\left(\int_s^t\int_{B(x_0,1)}|\partial_{\tau}u^{\varepsilon}_K|^2dx\right)^{1/2}\left(\int_s^t\int_{B(x_0,1)}\psi_{x_0}^2dxd\tau\right)^{1/2}\\
&\leq&C(n,m,t)\|\nabla u^{\varepsilon}_0\|_{L^2_{loc}(\R^n)}\sqrt{t-s}\|\psi_{x_0}\|_{L^2(\R^n)},
\end{eqnarray*}
where we used the a priori uniform bound (\ref{equ:3:bis}) in the last line. Now, by letting $s$ go to $0$, one gets:

\begin{eqnarray*}
\left|\int_{\R^n}<u^{\varepsilon}_K(t)-u^{\varepsilon}_0,\psi_{x_0}>dx\right|&\leq&C(n,m,t)\|\nabla u^{\varepsilon}_0\|_{L^2_{loc}(\R^n)}\sqrt{t}\|\psi_{x_0}\|_{L^2(\R^n)},
\end{eqnarray*}
as $u^{\varepsilon}_K(t)$ converges to $u_0$ weakly as $t$ goes to $0$. By letting $K$ go to $+\infty$, one has:
\begin{eqnarray*}
\left|\int_{\R^n}<u^{\varepsilon}(t)-u^{\varepsilon}_0,\psi_{x_0}>dx\right|&\leq&C(n,m,t)\|\nabla u^{\varepsilon}_0\|_{L^2_{loc}(\R^n)}\sqrt{t}\|\psi_{x_0}\|_{L^2(\R^n)},
\end{eqnarray*}
which proves the expected convergence as $t$ goes to $0$. This ends the proof of the claim.

The fact that $U^{\varepsilon}$ is regular off a singular closed (hence compact by (\ref{equ:1:bis})) set of finite $(n-2)$ Hausdorff dimensional measure follows from [Sect. III, \cite{Chen-Struwe}] and \cite{Cheng-Xiaoxi}. Finally, the fact that $u^{\varepsilon}$ solves the harmonic map flow follows from [Sect. III, \cite{Chen-Struwe}] as well.

The same strategy can be applied now to the sequence of expanding solutions $(u^{\varepsilon})_{\varepsilon\in(0,1)}$ of the Harmonic map flow by using (\ref{equ:1:bis}), (\ref{equ:2:bis}), (\ref{equ:3:bis}) together with (\ref{app-lip-maps}).  

The remaining statement to prove concerns the convergence rate at infinity (\ref{conv-rate}). By using the evolution equation (\ref{eq-HMP-Stat}), it is sufficient to prove the following claim:
\begin{claim}\label{claim-der}
Let $U$ be a solution of (\ref{eq-HMP-Stat}). Assume $U$ is smooth on $B(x_0,2r)$ for some positive radius $r$. Then,
\begin{eqnarray*}
\sup_{B(x_0,r/2)}|\nabla^2U|^2\leq C\left(1+\frac{1}{r^2}+\frac{|x_0|}{r}+\sup_{B(x_0,r)}|\nabla U|+\sup_{B(x_0,r)}|\nabla U|^2\right)\sup_{B(x_0,r)}|\nabla U|^2,
\end{eqnarray*}
where $C$ is a positive constant independent of $U$, $r$ and $x_0$.
\end{claim}
\begin{proof}[Proof of Claim \ref{claim-der}]
We proceed in the spirit of Shi's estimates for the Ricci flow \cite{Shi-Def}. We compute the evolution equation of the first two derivatives of $U$
\begin{eqnarray*}
\Delta_f\nabla U&=&-\frac{\nabla U}{2}+A(U)\ast\nabla^2U\ast \nabla U+D_UA\ast\nabla U^{*3},\\
\Delta_f\nabla^2U&=&-\nabla^2U+\nabla^3U\ast\nabla U+\nabla^2U^{*2}+\nabla^2U\ast\nabla U^{*2}+\nabla U^{*4}.
\end{eqnarray*}
In particular, by using Young's inequality
\begin{eqnarray*}
\Delta_f|\nabla U|^2&\geq&2|\nabla^2 U|^2-|\nabla U|^2-|\nabla^2U||\nabla U|^2-c|\nabla U|^4\\
&\geq&|\nabla^2 U|^2-c(1+|\nabla U|^2)|\nabla U|^2,\\
\end{eqnarray*}
where $c$ denotes a positive constant depending on the geometry of $N$ only that can vary from line to line. Similarly, one has:
\begin{eqnarray*}
\Delta_f|\nabla^2 U|^2&\geq&|\nabla^3 U|^2-c(1+|\nabla U|^2)|\nabla^2U|^2-c|\nabla^2 U|^3-c(1+|\nabla U|^2)|\nabla U|^4.\\
\end{eqnarray*}
Now, let $a$ be a positive constant to be defined later and consider the auxiliary function $F:=(|\nabla U|^2+a^2)|\nabla^2 U|^2.$ The function $F$ satifies the following differential inequality
\begin{eqnarray*}
\Delta_fF&\geq& |\nabla^2U|^4-c(1+|\nabla U|^2)F-8|\nabla^2U|^2|\nabla U||\nabla^3U|\\
&&+|\nabla^3U|^2(|\nabla U|^2+a^2)-c|\nabla^2U|^3(|\nabla U|^2+a^2)\\
&&-c|\nabla U|^4(|\nabla U|^2+a^2)(1+|\nabla U|^2)\\
&\geq&\frac{1}{2}|\nabla^2U|^4-c(1+|\nabla U|^2)F+|\nabla^3U|^2(a^2-c|\nabla U|^2)\\
&&-c|\nabla^2U|^3(|\nabla U|^2+a^2)-c|\nabla U|^4(|\nabla U|^2+a^2)(1+|\nabla U|^2).
\end{eqnarray*}
If $a^2:=c\sup_{B(x_0,1)}|\nabla U|^2$ where $c$ is a positive constant sufficiently large, then
\begin{eqnarray*}
\Delta_fF&\geq&\frac{F^2}{2a^4}-c(1+a^2)F-ca^{-1}F^{3/2}-c(1+a^2)a^6\\
&\geq&\frac{F^2}{4a^4}-c(1+a^2)F-c(1+a^2)a^6,
\end{eqnarray*}
by Young's inequality. 

Let $\phi:\R^n\rightarrow[0,1]$ be a smooth positive function with compact support defined by $\phi(x):=\psi(r_{x_0}/r)$ where $r_{x_0}(x):=|x-x_0|$ and where $r>0$ and $\psi:[0,+\infty[\rightarrow[0,1]$ is a smooth positive function satisfying
\begin{eqnarray*}
\psi\arrowvert_{[0,1/2]}\equiv 1,\quad \psi\arrowvert_{[1,+\infty[}\equiv 0,\quad \psi'\leq 0,\quad \frac{\psi'^2}{\psi}\leq c,\quad \psi''\geq -c.
\end{eqnarray*}

 Define the (last) auxiliary function $G:=\phi F$ and consider a point $x_1\in B(x_0,1)$ such that $G(x_1)=\max_{B(x_0,1)}G$. By the previous differential inequality satisfied by $F$ evaluated at $x_1$ together with the maximum principle
\begin{eqnarray}
0&=&\nabla G=F\nabla \phi+\phi\nabla F,\\
0&\geq& \phi\Delta_fG\geq \frac{G^2}{4a^4}-ca^2G-c(1+a^2)a^6-2G\frac{\arrowvert\nabla\phi\arrowvert^2}{\phi}+G(\Delta \phi+\langle \nabla f,\nabla \phi\rangle).\label{inequ-der-shi}
\end{eqnarray}
Now,
\begin{eqnarray*}
\nabla \phi=\frac {\psi'}{r}\nabla r_{x_0},\quad\Delta \phi=\frac{\psi''}{r^2}+\frac{\psi'}{r}\Delta r_{x_0}.
\end{eqnarray*}
Hence,
\begin{eqnarray*}
2\frac{\arrowvert\nabla\phi\arrowvert^2}{\phi}-\Delta_f \phi=\frac{1}{r^2}\left[\frac{2\psi'^2}{\psi}-\psi''\right]-\frac{\psi'}{r}(\langle \nabla f,\nabla r_{x_0}\rangle+\Delta r_{x_0}).
\end{eqnarray*}
On the other hand,
\begin{eqnarray*}
\Delta r_{x_0}\leq\frac{n-1}{r}, \quad\mbox{on $B(x_0,r)\backslash B(x_0,r/2)$.}
\end{eqnarray*}
Coming back to inequality (\ref{inequ-der-shi}), one gets
\begin{eqnarray*}
0&\geq& \frac{G^2}{4a^4}-c\left(a^2+\frac{1}{r^2}+\frac{\sup_{B(x_0,r)}|\nabla f|}{r}\right)G-a^6\\
&\geq&\frac{G^2}{4a^4}-c\left(1+a^2+\frac{1}{r^2}+\frac{|x_0|}{r}\right)G-c(1+a^2)a^6,\\
\end{eqnarray*}
by the very definition of $f$ together with the triangular inequality.
The expected estimate on the second derivatives of $U$ follows immediately.

\end{proof}

We are now in a position to prove the convergence rate as stated in (\ref{conv-rate}). Indeed, by Theorem \ref{eps-reg-Chen-Str} together with Claim \ref{claim-der} applied to a point $x_0\in\R^n$ sufficiently far from the origin $0\in\R^n$ and to a radius $r:=|x_0|/2$, one gets in particular that the Laplacian of $U$ decays at least linearly at infinity. Consequently, by equation (\ref{eq-HMP-Stat}), the radial derivative of $U$ decays at least quadratically. By integrating along radial lines, $U$ approaches $u_0$ at a linear rate: $$|U(x)-u_0(x/|x|)|=\textit{O}(|x|^{-1}),$$ for every $x$ far from the origin. 

\end{proof}

\begin{rk}
It does not seem straightforward to improve the convergence rate in case $u_0$ is at least $C^3_{loc}(\R^n\setminus\{0\}).$ The main reason is the lack of an $\varepsilon$-regularity theorem that detects the smoothness of the map $u_0$ at infinity.
\end{rk}

 \section{Taylor expansions at infinity for expanders of the harmonic map flow }\label{section-Taylor}
We gather necessary conditions at infinity on an expanding solution of the Homogeneous Ginzburg-Landau flow (\ref{eq-GL-K}) or the harmonic map flow smoothly coming out of a $0$-homogeneous map $u_0:\mathbb{S}^{n-1}\rightarrow\mathbb{S}^{m-1}\subset\R^m$ in $C^{\infty}(\R^n\setminus\{0\})$. A similar treatment could be done for the Homogeneous Chen-Struwe flow for a general target closed manifold $(N,g)$ isometrically embedded in some Euclidean space $\R^m$.

Let $U$ be an expanding solution to the Homogeneous Ginzburg-Landau flow with parameter $K>0$ coming out of the map $u_0$.

Let us assume that there are smooth maps $u_i:\mathbb{S}^{n-1}\rightarrow\R^m$, $i=1,..., k$, such that
\begin{eqnarray*}
U(x)=\sum_{i=0}^k\frac{u_i(x/|x|)}{|x|^{2i}}+\textit{O}(|x|^{-2k-2}),
\end{eqnarray*}
as $x$ goes to $+\infty$ for every nonnegative integer $k$. This expansion is assumed to hold in the smooth sense. Then, on one hand,
\begin{eqnarray*}
\Delta_fU&=&\sum_{i=0}^k\left(|x|^{-2i-2}\Delta_{\mathbb{S}^{n-1}} u_i+\Delta_f(|x|^{-2i})u_i\right)+\textit{O}(|x|^{-2k-2})\\
&=&\sum_{i=0}^k\left(|x|^{-2i-2}\Delta_{\mathbb{S}^{n-1}} u_i+\left(2i(2(i+1)-n)|x|^{-2}-i\right)|x|^{-2i}u_i\right)+\textit{O}(|x|^{-2k-2})\\
&=&\sum_{i=1}^k|x|^{-2i}\left(\Delta_{\mathbb{S}^{n-1}} u_{i-1}-iu_i+2(i-1)(2i-n)u_{i-1}\right)+\textit{O}(|x|^{-2k-2}).
\end{eqnarray*}
On the other hand,
\begin{eqnarray*}
(1-|U|^2)U&=&\left(1-\left|\sum_{i=0}^k|x|^{-2i}u_i+\textit{O}{(|x|^{-2k-2})}\right|^2\right)\left(\sum_{i=0}^k|x|^{-2i}u_i+\textit{O}{(|x|^{-2k-2})}\right)\\
&=&-\left(\sum_{i=1}^{k}|x|^{-2i}\left(\sum_{j=0}^i<u_j,u_{i-j}>\right)\right)\left(\sum_{i=0}^k|x|^{-2i}u_i\right)+\textit{O}{(|x|^{-2k-2})}\\
&=&-\sum_{i=0}^k|x|^{-2i}\left(\sum_{j=0}^ia_ju_{i-j}\right)+\textit{O}{(|x|^{-2k-2})},
\end{eqnarray*}
where 
\begin{eqnarray*}
a_0:=0,\quad a_i:=\sum_{j=0}^i<u_j,u_{i-j}>,\quad i\geq 1.
\end{eqnarray*}

Therefore, by identifying terms by terms, for $i\geq 1$,
\begin{eqnarray*}
\Delta_{\mathbb{S}^{n-1}} u_{i-1}-iu_i+2(i-1)(2i-n)u_{i-1}-K\sum_{j=1}^ia_ju_{i-j}=0,
\end{eqnarray*}
which gives for $i=1$:
\begin{eqnarray*}
\Delta_{\mathbb{S}^{n-1}} u_{0}-u_1-2K<u_1,u_0>u_0=0,
\end{eqnarray*}
that is: $$u_1=\Delta_{\mathbb{S}^{n-1}} u_{0}-\frac{2K}{2K+1}<\Delta_{\mathbb{S}^{n-1}} u_{0},u_0>u_0.\label{first-coeff-ginz-land}$$
In general, if $i\geq 2$, one has:
\begin{eqnarray*}
iu_i+2K<u_i,u_0>u_0&=&\Delta_{\mathbb{S}^{n-1}} u_{0}+2(i-1)(2i-n)u_{i-1}\\
&&-K\sum_{j=1}^{i-1}a_ju_{i-j}-K\sum_{j=1}^{i-1}<u_j,u_{i-j}>u_0,
\end{eqnarray*}
which determines $u_i$.\\

Let $U$ be an expanding solution to the harmonic map flow $K>0$ coming out of the map $u_0$. Then, since the maps $(u_i)_{i\geq 0}$ are spherical, 
\begin{eqnarray*}
|\nabla U|^2&=&\left|\nabla\left(\sum_{i=0}^k|x|^{-2i}u_i\right)\right|^2+\textit{O}(|x|^{-2k-2})\\
&=&\left|\sum_{i=0}^k(-2i)|x|^{-2i-1}u_i\right|^2+\left|\sum_{i=0}^k|x|^{-2i-1}(\nabla^{\mathbb{S}^{n-1}} u_i)\right|^2+\textit{O}(|x|^{-2k-2})\\
&=&\sum_{i=0}^{k-1}|x|^{-2i-2}\left(\sum_{j=0}^i4j(i-j)<u_j,u_{i-j}>+<\nabla^{\mathbb{S}^{n-1}} u_j,\nabla^{\mathbb{S}^{n-1}} u_{i-j}>\right)+\textit{O}(|x|^{-2k-2}),
\end{eqnarray*}
which implies:
\begin{eqnarray*}
|\nabla U|^2U&=&\sum_{i=0}^k|x|^{-2i}\left(\sum_{l=0}^ib_{l}u_{i-l}\right)+\textit{O}(|x|^{-2k-2}),
\end{eqnarray*}
where,
\begin{eqnarray*}
b_0:=0,\quad b_{i+1}:=\sum_{j=0}^i4j(i-j)<u_j,u_{i-j}>+<\nabla^{\mathbb{S}^{n-1}} u_j,\nabla^{\mathbb{S}^{n-1}} u_{i-j}>,\quad i\geq 0.
\end{eqnarray*}
 By identification,
 \begin{eqnarray*}
\Delta_{\mathbb{S}^{n-1}} u_{i-1}-iu_i+2(i-1)(2i-n)u_{i-1}=-\sum_{l=0}^ib_{l}u_{i-l},\quad i\geq 1.
\end{eqnarray*}
For instance, if $i=1$,
\begin{eqnarray}
u_1=\Delta_{\mathbb{S}^{n-1}} u_{0}+|\nabla^{\mathbb{S}^{n-1}} u_0|^2u_0, \label{first-coeff-HMF}
\end{eqnarray}
which can be understood as the formal limit as $K$ goes to $+\infty$ of the sequence of the corresponding coefficients $(u_1^K)_{K>0}$ of the Taylor expansion derived previously for the Ginzburg-Landau equation with parameter $K$. Indeed, as $K$ goes to $+\infty$, one gets by (\ref{first-coeff-ginz-land}) that:
\begin{eqnarray*}
\lim_{K\rightarrow +\infty}u_1^K&=&\Delta_{\mathbb{S}^{n-1}} u_{0}-<\Delta_{\mathbb{S}^{n-1}} u_{0},u_0>u_0\\
&=&\Delta_{\mathbb{S}^{n-1}} u_{0}+|\nabla^{\mathbb{S}^{n-1}} u_0|^2u_0,
\end{eqnarray*}
which is exactly formula (\ref{first-coeff-HMF}) since $0=\Delta_{\mathbb{S}^{n-1}} |u_{0}|^2=2<\Delta_{\mathbb{S}^{n-1}}u_0,u_0>+2|\nabla^{\mathbb{S}^{n-1}} u_0|^2.$ 

The same holds for the other coefficients $(u_k)_{k\geq 1}$.

Moreover, one can check that if $u_0$ is harmonic, i.e. if $\Delta_{\mathbb{S}^{n-1}} u_{0}+|\nabla^{\mathbb{S}^{n-1}} u_0|^2u_0=0$ then all the other terms vanish: $u_k=0$, for all $k\geq 1$.

In particular, we get the following corollary:

\begin{coro}\label{coro-conv-rate-harm}
Let $u_0:\R^n\rightarrow \mathbb{S}^{m-1}$ be a $0$-homogeneous map of the harmonic map flow in $C^{\infty}(\R^n\setminus\{0\})$ and 
let $U:\R^n\rightarrow \mathbb{S}^{m-1}$ be an expanding solution of the harmonic map flow smoothly  coming out of $u_0$. Then the convergence rate of $U$ at infinity is faster than any polynomial rate.
\end{coro}

\begin{rk}
As in \cite{Der-Asy-Com-Egs}, it can be shown that the convergence rate of a smooth expanding solution $u$ to its initial condition $u_0$ is exactly $\textit{O}\left(r^{-n}e^{-r^2/4}\right)$ if $u_0$ is harmonic, hence much faster than what Corollary \ref{coro-conv-rate-harm} predicts. This decay reveals the role of the Hermite functions at infinity: they are intimately connected to the weighted Laplacian $\Delta_f$. 
\end{rk}

\bibliographystyle{alpha.bst}
\bibliography{bib-HMF-v2-arxiv}

\def\cprime{$'$}
\begin{thebibliography}{{Der}18}

\bibitem[BB11]{Bie-Biz}
Pawe\l Biernat and Piotr Biz\'on.
\newblock Shrinkers, expanders, and the unique continuation beyond generic
  blowup in the heat flow for harmonic maps between spheres.
\newblock {\em Nonlinearity}, 24(8):2211--2228, 2011.

\bibitem[CD16]{Con-Der}
R.~J. {Conlon} and A.~{Deruelle}.
\newblock {Expanding K\"ahler-Ricci solitons coming out of K\''ahler cones}.
\newblock {\em ArXiv e-prints}, July 2016.

\bibitem[Che89]{Chen-Har-Map}
Yun~Mei Chen.
\newblock The weak solutions to the evolution problems of harmonic maps.
\newblock {\em Math. Z.}, 201(1):69--74, 1989.

\bibitem[Che91]{Cheng-Xiaoxi}
Xiaoxi Cheng.
\newblock Estimate of the singular set of the evolution problem for harmonic
  maps.
\newblock {\em J. Differential Geom.}, 34(1):169--174, 1991.

\bibitem[CS89]{Chen-Struwe}
Yun~Mei Chen and Michael Struwe.
\newblock Existence and partial regularity results for the heat flow for
  harmonic maps.
\newblock {\em Math. Z.}, 201(1):83--103, 1989.

\bibitem[Der16]{Der-Smo-Pos-Met-Con}
Alix Deruelle.
\newblock Smoothing out positively curved metric cones by {R}icci expanders.
\newblock {\em Geom. Funct. Anal.}, 26(1):188--249, 2016.

\bibitem[{Der}17]{Der-Asy-Com-Egs}
A.~{Deruelle}.
\newblock Asymptotic estimates and compactness of expanding gradient ricci
  solitons.
\newblock {\em Annali della Scuola Normale Superiore di Pisa}, Vol.
  XVII:485--530, 2017.

\bibitem[{Der}18]{Der-rel-Ent-HMF}
Alix {Deruelle}.
\newblock {A relative entropy for expanders of the Harmonic map flow}.
\newblock {\em arXiv e-prints}, page arXiv:1807.00140, Jun 2018.

\bibitem[GGM17]{Ger-Gho-Miu}
Pierre Germain, Tej-Eddine Ghoul, and Hideyuki Miura.
\newblock On uniqueness for the harmonic map heat flow in supercritical
  dimensions.
\newblock {\em Comm. Pure Appl. Math.}, 70(12):2247--2299, 2017.

\bibitem[GR11]{Ger-Rup}
Pierre Germain and Melanie Rupflin.
\newblock Selfsimilar expanders of the harmonic map flow.
\newblock {\em Ann. Inst. H. Poincar\'e Anal. Non Lin\'eaire}, 28(5):743--773,
  2011.

\bibitem[Jv14]{Jia-Sverak}
Hao Jia and Vladim\'ir \v{S}ver\'ak.
\newblock Local-in-space estimates near initial time for weak solutions of the
  {N}avier-{S}tokes equations and forward self-similar solutions.
\newblock {\em Invent. Math.}, 196(1):233--265, 2014.

\bibitem[KL12]{Koc-Lam-Rou}
Herbert Koch and Tobias Lamm.
\newblock Geometric flows with rough initial data.
\newblock {\em Asian J. Math.}, 16(2):209--235, 2012.

\bibitem[KP99]{Kra-Par}
Steven~G. Krantz and Harold~R. Parks.
\newblock {\em The geometry of domains in space}.
\newblock Birkh\"{a}user Advanced Texts: Basler Lehrb\"{u}cher. [Birkh\"{a}user
  Advanced Texts: Basel Textbooks]. Birkh\"{a}user Boston, Inc., Boston, MA,
  1999.

\bibitem[LW08]{Lin-Wang}
Fanghua Lin and Changyou Wang.
\newblock {\em The analysis of harmonic maps and their heat flows}.
\newblock World Scientific Publishing Co. Pte. Ltd., Hackensack, NJ, 2008.

\bibitem[LW17]{Lott-Wil}
John Lott and Patrick Wilson.
\newblock Note on asymptotically conical expanding {R}icci solitons.
\newblock {\em Proc. Amer. Math. Soc.}, 145(8):3525--3529, 2017.

\bibitem[Shi89]{Shi-Def}
Wan-Xiong Shi.
\newblock Deforming the metric on complete {R}iemannian manifolds.
\newblock {\em J. Differential Geom.}, 30(1):223--301, 1989.

\bibitem[Wan11]{Wang-BMO}
Changyou Wang.
\newblock Well-posedness for the heat flow of harmonic maps and the liquid
  crystal flow with rough initial data.
\newblock {\em Arch. Ration. Mech. Anal.}, 200(1):1--19, 2011.

\end{thebibliography}

\end{document}